\documentclass[10pt]{article}
\usepackage[english]{babel}
\usepackage{amssymb}
\usepackage{amsfonts}
\usepackage{amsmath}
\usepackage{amsthm,thmtools,xcolor}
\usepackage{mathtools}
\usepackage[mathscr]{euscript}
\let\euscr\mathscr 
\usepackage{graphicx}
\usepackage{caption}
\usepackage{bm}
\usepackage{enumerate}
\usepackage[colorlinks=true, allcolors=blue]{hyperref}
\usepackage{mathrsfs}
\usepackage{algorithm}
\usepackage{algpseudocode}
\usepackage{dsfont}
\usepackage{authblk}
\usepackage{abstract}

\theoremstyle{plain}
\newtheorem{theorem}{Theorem}[section]
\newtheorem{lemma}[theorem]{Lemma}
\newtheorem{corollary}[theorem]{Corollary}
\newtheorem{proposition}[theorem]{Proposition}
\theoremstyle{definition}
\newtheorem{definition}[theorem]{Definition}
\newtheorem{example}[theorem]{Example}
\theoremstyle{remark}
\newtheorem{remark}{Remark}

\declaretheoremstyle[
  headfont=\color{red}\normalfont\bfseries,
  ]{coloured}

\newcommand{\Rd}{\mathbb{R}^d}
\newcommand{\Sd}{\mathbb{S}^{d-1}}
\newcommand{\R}{\mathbb{R}}

\newcommand{\Pset}{\euscr{P}}
\newcommand{\Shw}{\mathcal{S}}
\newcommand{\mb}{\mathbf}

\newcommand{\bom}{\bm{\omega}}

\def\mb#1{\mathbf{#1}}
\def\mc#1{\mathscr{#1}}
\def\mbb#1{\mathbb{#1}}

\title{Multidimensional Dickman distribution and operator selfdecomposability}

\author[1]{Anastasiia S. Kovtun}
\author[1]{Nikolai N. Leonenko}
\author[1]{Andrey Pepelyshev}
\affil[1]{School of Mathematics, Cardiff University}

\begin{document}

\maketitle
\begin{abstract}

The one-dimensional Dickman distribution arises in various stochastic models across number theory, combinatorics, physics, and biology. Recently, a definition of the multidimensional Dickman distribution has appeared in the literature, together with its application to approximating the small jumps of multidimensional L\'evy processes. In this paper, we extend this definition to a class of vector-valued random elements, which we characterise as fixed points of a specific affine transformation involving a random matrix obtained from the matrix exponential of a uniformly distributed random variable.  We prove that these new distributions possess the key properties of infinite divisibility and operator selfdecomposability. Furthermore, we identify several cases where this new distribution arises as a limiting distribution.
\end{abstract}


\section{Introduction}
The one-dimensional Dickman distribution, as defined in \cite{PW}, appears as a limiting law in several settings. For instance, it arises as the long-time limit of a shot-noise process with fixed signal amplitude and exponential propagation \cite{takacs1954secondary}, and in a model of the atomic cascade process \cite{ChamayouSchorr1975}. It can also be characterised as a fixed point of an affine transformation with random coefficients, or, equivalently, as the stationary solution of a random difference equation \cite{vervaat1979stochastic}.
Its relative, the max-Dickman distribution, initially arose in number theory \cite{Dickman30, deBruijn1951number}, and later appeared in a combinatorial setting \cite{Goncharov44}.   Recently, a number of new applications of the Dickman distribution have been discovered in random graph theory \cite{PW} as well as in various limiting schemes \cite{pinsky18,caravenna2019dickman,bhattacharjee2019dickman,bhattacharjee2020convergence,GMP}. The application of Dickman distribution in the theory of non-local operators was discussed in \cite{gupta2024generalized}. It was also shown to be useful in approximating the small jumps of positive L\'evy processes in cases when Brownian approximation fails \cite{covo2009approximations}. Some generalisations of the Dickman distribution have been considered in \cite{arratia2003logarithmic, pinsky18, bhattacharjee2019dickman, grahovac2025dickman}
 For simulation from the Dickman distribution, we refer to \cite{devroye2010simulating,fill2010perfect,dassios2019exact}.
For the historical review on the Dickman distribution, see \cite{PW, MP}. 

Recently, the generalisation of the one-dimensional Dickman distribution to the multidimensional case has appeared in \cite{bhattacharjee2020convergence} together with its application towards approximating small jumps of multidimensional L\'evy processes \cite{grabchak2024smalljumpslevyprocesses}. In this paper, we propose to extend the definition given in \cite{bhattacharjee2020convergence} to a class of vector-valued random elements, which we define, in analogy to the one-dimensional version, as fixed points of a certain affine transformation with random coefficients. In particular, the affine transformation we consider is induced by translating a given point by a random vector and further multiplying it by a factor defined as the operator (matrix) exponential of a uniformly distributed random variable. We refer to the elements of this class as the operator Dickman distributions. The motivation for this extension initially comes from the property of operator selfdecomposability, which we shall demonstrate holds for members of our class. It is also supported by some applications analogous to the classical applications of the one-dimensional Dickman distribution. 

The present paper is structured as follows. We recall the definition and key properties of the one-dimensional Dickman distribution in Section \ref{sec:Univ_D}. In Section \ref{sec:Multiv_D}, we first consider the multidimensional Dickman distribution as defined in \cite{bhattacharjee2020convergence}, and then introduce a broader class of operator Dickman distributions, highlighting some of their key properties. Some applications of the operator Dickman distributions are considered in Section \ref{sec:appl}. The simulations from the operator Dickman distribution are demonstrated in Section \ref{sec: simul}.

\section{One-dimensional Dickman distribution}\label{sec:Univ_D}

The one-dimensional Dickman distribution with parameter $\theta>0$ can be defined as the distribution of a random variable $X_D$ satisfying the equation 
\begin{equation}\label{eq: GD_def}
X_D \overset{d}{=} U^{1/\theta} ( 1 +X'_D),
\end{equation}
where the symbol ``$\overset{d}{=}$'' denotes the equality in distribution, $X'_D \overset{d}{=} X_D$, $U$ has the uniform distribution on $[0,1]$ and is independent of $X'_D$, see e.g. \cite{PW}. 
Equivalently, the random variable $X_D$ is the fixed point (in a distributional sense) of the random affine transformation $\cdot \mapsto U^{1/\theta}(1+\cdot)$, where $U$ is as above. 
We will further denote $GD_\theta$, the Dickman distribution with parameter $\theta>0$.

The distribution $GD_\theta$ is absolutely continuous for all $\theta>0$, and its density can be represented as 
\begin{equation*}
    f_{\theta}(x)=\frac {e^{-\gamma\theta}} {\Gamma(\theta)}\rho_{\theta}(x), ~x\in\mathbb{R},
\end{equation*}
where $\Gamma(\cdot)$ is the gamma function, $\gamma=-\Gamma'(1)\approx0.5772 $ is Euler's constant, and the function $\rho_{\theta}(x)$ satisfies the difference-differential equation
\begin{equation}\label{eq: Dickman_ODE}
\begin{aligned}
    x \rho'_{\theta}(x) +(1-\theta)\rho_{\theta}(x)+\theta\rho_{\theta}(x-1)=0, \quad  x>1,
\end{aligned}
\end{equation}
with initial conditions 
\begin{equation*}
\begin{aligned}
    &\rho_{\theta}(x) =0, \quad  x\leq 0,\\
    &\rho_{\theta}(x) =x^{\theta-1}, \quad 0<x\leq 1.
\end{aligned}
\end{equation*}

In the case of $\theta=1$, the function $\rho_1(\cdot)$ is known as the Dickman function. It occurred in the work of Karl Dickman \cite{Dickman30} in the number-theoretical context from which it gained its name. However, it can be found in even earlier Ramanujan's unpublished notes \cite[Ch. 8]{andrews_r_15}. 
The equation \eqref{eq: Dickman_ODE} and its generalisations have been studied in \cite{wheeler1990two}.

The density $f_{\theta}(x)$ of the $GD_\theta$ distribution has been an object of independent studies, and we recall here a few results related to it. 
The asymptotic behaviour of the density $f_{\theta}(x)$ for $\theta=1$ has been established in \cite{deBruijn1951asymptotic} and for general $\theta>0$ in \cite[Lem. 4.7.9]{vervaat1972success}. It was also proven in \cite[Thm. 4.7.7]{vervaat1972success} that the density $f_\theta(x)$ can be expressed explicitly as follows
\begin{equation*}
\begin{aligned}
    f_{\theta}(x)=\frac {e^{-\gamma\theta}} {\Gamma(\theta)}\left( x^{\theta-1}+\sum_{k=1}^{[x]-1}\frac {(-\theta)^k}  {k!} \idotsint\limits_{\substack{t_1,\ldots, t_k> 1 \\  t_1+\cdots+t_k\leq x}}(x-t_1-\cdots - t_k)^{\theta-1}\frac {dt_1\ldots dt_k} {t_1 \cdot\ldots\cdot t_k}\right),
\end{aligned}
\end{equation*}
where $x>0$, $[\cdot]$ is the integer part function;
see \cite[Thm. 1]{griffiths1988distribution} and \cite[Prop. 1]{gupta2024generalized} for similar representations.
Moreover, it was shown in \cite{caravenna2019dickman} that the following recurrent relation holds
\begin{equation*}
    f_{\theta}(x)=\begin{cases}
        \frac {e^{-\gamma\theta}} {\Gamma(\theta)} x^{\theta-1},& 0<x\leq 1,\\
        \frac {e^{-\gamma\theta}} {\Gamma(\theta)} x^{\theta-1}- \theta x^{\theta-1}\int_0^{x-1}\frac {f_{\theta}(z)} {(1+z)^{\theta}}dz,& x>1.
    \end{cases}
\end{equation*}

Note also that it follows from \eqref{eq: GD_def} that $X_D$ can be represented as 
\begin{equation}\label{eq: perp}
    X_D \overset{d}{=} U_1^{1/\theta}+(U_1U_2)^{1/\theta}+(U_1U_2U_3)^{1/\theta}+(U_1U_2U_3U_4)^{1/\theta}+\cdots,
\end{equation}
where $\{U_k\}_{k=1}^{\infty}$ is the sequence of mutually independent identically distributed random variables with uniform distribution on $[0,1]$, see, e.g., \cite{vervaat1979stochastic}. The series of the form \eqref{eq: perp} is commonly referred to as a perpetuity and is usually related to a random difference equation, see the equation \eqref{eq:randomAR} below. 
The random difference equations, which allow modelling discrete processes with a random discounting coefficient and random input at each time step, have been studied in a series of works, see, for example, \cite{kesten1973random,vervaat1979stochastic,embrechts1994perpetuities}.

The characteristic function of $X_{D}$ distributed according to $GD(\theta)$ is given by
\begin{equation}\label{e:CF_D}
\psi_D(z) = \mathbb{E} e^{izX_{D}} = \exp \left\{  \theta \int\limits_0^1 (e^{izu}-1) \frac{du}{u}  \right\},~ z\in \mathbb{R},
\end{equation}
and it follows from \eqref{e:CF_D} that the distribution $GD_\theta$ is infinitely divisible with the L\'evy measure 
\begin{equation*}
    \upsilon(du)=\frac {\theta} u \mathds{1}_{(0,1)}(u)du.
\end{equation*}
The $GD_\theta$ distribution is also self-decomposable with the canonical function $k(x)$ given by $k(x)=\theta\mathds{1}_{(0,1)}(x)$ \cite[Cor. 15.11]{sato1999levy}.

\section{Dickman distributions over \texorpdfstring{$\mathbb{R}^d$}{R\textasciicircum d}}\label{sec:Multiv_D}

The multidimensional Dickman distribution (as in Definition \ref{def:MD} below) was first introduced in \cite{bhattacharjee2020convergence} and further studied in \cite{grabchak2024representation}, see also \cite{grabchak2024smalljumpslevyprocesses, grabchak2026simulation}. In this section, we generalise the definition given there to a new class of distributions, which we will refer to as the operator Dickman distributions. We further study its properties. In particular, we derive the expression for the characteristic function of the operator Dickman distributions and consequently show that they constitute a subset of infinitely divisible distributions over $\Rd$. We also demonstrate that the class of operator Dickman distributions forms a subset of the set of operator selfdecomposable distributions.  Some specific examples of the distributions arising from this class are considered at the end of the section. 

 First, we introduce some notation which will be used throughout the paper.
 For elements $\mb x$ and $\mb y$ of the Euclidean space $\R^d$, $d\geq 1$, we denote their inner product by $\mb x\cdot\mb y$ and the Euclidean norm of $\mb x \in \Rd$ as $|\mb x|$. The set $\mbb S^{d-1}=\{\mb x\in \R^d: |\mb x|=1\}$ is the unit sphere in $\R^d$. For any topological space $S$, let $\mc B(S)$ denote the Borel sigma-algebra over $S$. Let also $\R_+$ denote the set of positive real numbers, and $\mathbb{C}$ denote the set of complex numbers.
 We will also use $\delta$ and $\mathds{1}$ as standard notation for the Dirac delta function and the indicator function, respectively; $a\wedge b$ denotes $\min\{a,b\}$ for $a,b\in \R$.
 
 The linear transformations of $\R^d$ to $\R^d$ are usually identified with their matrix representations, and the set of all $d\times d$ matrices is denoted by $\mc M$. Additionally,  $\mc M_+\subset\mc M$ ($\mc M_-\subset\mc M$) denotes the set of all elements of $\mc M$ whose eigenvalues have positive (negative) real parts. Let $I\in\mc{M}$ and $O\in \mc M$ denote, respectively, the identity and zero transformations, let also $\mathbb{I}=\{cI, c>0\}$ be a set of positive multiples of the identity transformation.  For any transformation $Q\in\mc M$, let $Q^*$ denote its adjoint. For any $Q\in \mc M$, $t\in \R$ and $u\in \R_+$, define the matrix exponential as follows
 \begin{equation*}
     e^{tQ}\coloneq \sum_{k=0}^\infty \frac {t^k} {k!}Q^k \quad \text{and}\quad u^{Q}\coloneq e^{Q\log u}.
 \end{equation*} Let $\|Q\| \coloneq\sup_{|\mb x|\leq1}|Q\mb x|$ denote the operator norm of $Q\in\mc M$.
 For any measure $\vartheta$ on $\Rd$ and any measurable map $T: \Rd\mapsto \Rd$, denote $T\vartheta$ the measure given by $T\vartheta(A)\coloneq \vartheta (T^{-1}(A))$, $A\in \mc B(\Rd)$.
 
 Let $\euscr{P}(\R^d)$ denote the set of probability measures over $\R^d$ and for any $\pi\in \Pset(\R^d)$ let $\hat{\pi}$ denote its characteristic function, i.e. $\hat{\pi}(\mb z)=\int\nolimits_{\R^d}e^{i\mb z\cdot\mb x}\pi(d\mb x)$. Let $H\subset\euscr{P}(\R^d)$ be the set of all probability measures $\nu$ such that 
 \begin{equation*}
     \int\nolimits_{|\mb x|>1}\log |\mb x|\nu(d\mb x)<\infty,
 \end{equation*} and let $\mc S\subset \euscr{P}(\R^d)$ be the set of all probability distributions concentrated on $d$-dimensional sphere, i.e. $\sigma\in \mc S$ whenever $\sigma(\mathbb{S}^{d-1})=1$. Note that $\mc S\subset H$.
 For a random vector $\mb X\in \R^d$, let $\mc L(\mb X)\in \Pset(\R^d)$ denote its probability distribution.
 
 We also denote $ID(\R^d)\subset \Pset(\R^d)$, the set of all infinitely divisible distributions. Recall that for any $\varrho\in  ID(\R^d)$, the characteristic function of $\varrho$ is given by the L\'evy-Khintchine formula 
 \begin{equation*}
     \hat{\varrho}(\mb z)=\exp\left[ i \mb a\cdot \mb z -\frac 1 2 \mb z \cdot A\mb z+\int\limits_{\R^d}(e^{i\mb z \cdot \mb x}-1-i\mb z \cdot \mb x\mathds{1}_{|\mb x|\leq 1})\ell(d\mb x)\right],\quad\mb z\in \Rd,
 \end{equation*}
 where $\mb a \in \R^d$, $A\in \mc M$ is a symmetric non-negative definite matrix and $\ell$ is the L\'evy measure on $\R^d$ satisfying $\ell(\{\mb 0\})=0$ and $\int_{\Rd}(1\wedge |\mb x|^2)\ell(d\mb x)<\infty$; the triplet $(\mb a, A, \ell)$ is the characteristic triplet of the distribution $\varrho$, see \cite{sato1999levy} for more details.

Let us recall the definition of the multidimensional Dickman distribution as introduced in \cite{bhattacharjee2020convergence}.
\begin{definition}[Multidimensional Dickman distribution]\label{def:MD}
    We call the distribution $\mbb D_{\theta,\sigma}=\mc L(\mb D)$ of the random vector $\mb D=\mb D_\theta^{(\sigma)}$ a multidimensional Dickman distribution if $\mb D$ satisfies the distributional equation
    \begin{equation}\label{eq: MD}
        \mb D\overset{d}{=}U^{1/\theta}(\mb D'+\mathbf{V}),
    \end{equation}
where $\theta>0$, $\mb D \overset{d}{=}\mb D'$, $U$ has the uniform distribution on $[0,1]$,  $\mc L(\mathbf{V})=\sigma \in \mc S$, and the random elements $U, \mb V, \mb D'$ are mutually independent.
\end{definition}

Next, we generalise this definition to a wider class of distributions by substituting the scalar function $u^{1/\theta}, \theta>0$, with a matrix exponential $u^Q, Q\in \mc M_+$, and a probability measure $\sigma\in \mc S$ with a probability measure $\nu\in H$.

\begin{definition}\label{def: opMD}
For any $Q\in  \mc M_+$ and $\nu \in  H$, the 
probability distribution $\mc D(Q,\nu )\in \euscr{P}(\R^d)$ is the distribution of a vector $\mb X$, which satisfies
\begin{equation}\label{eq: opMD}
    \mb X\overset{d}{=}U^{Q}(\mb X'+\mb W),
\end{equation}
where $\mb X\overset{d}{=}\mb X'$, $U$ is uniformly distributed over $[0,1]$, $\mc L (\mb W)=\nu$,  and the random elements $U,\mb X'$ and $\mb W$ are mutually independent.

Moreover, for any $\mc Q\subseteq\mc M_+$ and $\mc H\subseteq H$, let 
$\mc D(\mc Q,\mc H)$ be a set of distributions defined as 
\begin{equation*}
    \mc D(\mc Q,\mc H)=\{\mc D(Q, \nu), Q\in \mc Q, \nu\in \mc H\},
\end{equation*}
and the set $\mc D(\mc M_+, H)$ will be referred to as the class of \textit{operator Dickman distributions}.
\end{definition}

\begin{remark}
 Note that the random operator $U^Q$ is well-defined since the mapping $(s,Q)\mapsto s^Q$ is well-defined and jointly continuous for $s>0$ and $Q\in \mc M$.  Further details and properties of the exponential operators can be found in \cite[Ch. 2]{meerschaert2001limit}.  
\end{remark}

Notice that for any $Q \in \mc M_+$ the map $\gamma(a, \mb b): \cdot \mapsto a^Q(\cdot + \mb b)$ with $a\in \R_+$ and $\mb b\in\Rd $ defines an affine transformation on $\Rd$. Moreover the family $\gamma=\{\gamma(a, \mb b), a>0, \mb b\in\Rd\}$ forms a non-commutative group under the composition, which is isomorphic to the group $\R_+\times \Rd$ with group operation $*$ defined as $(a_1, \mb b_1)*(a_2, \mb b_2):= (a_1a_2, (a_1a_2)^Q\mb b_1+a_2^Q\mb b_2)$. We can now equip the group $\gamma$ with a probability measure defined as the product measure $Leb_1\times \nu$, where $Leb_1$ is the Lebesgue measure on $[0,1]$ and $\nu\in H$ is some probability measure. Therefore, the vector $\mb X$, defined in \eqref{eq: opMD}, can be viewed as the distributional fixed point of the random affine transformation described above. For $d=1$, this problem was studied in multiple works, see \cite{vervaat1979stochastic} for a short overview and references, and \cite{bhattacharjee2019dickman} for a generalisation to non-affine transformations.

As another point of view on Definition \ref{def: opMD}, we can consider the difference equation 
\begin{equation}\label{eq:randomAR}
    \mb Y_n=M_n \mb Y_{n-1}+\bm{\xi}_n,\quad n\geq 1,
\end{equation}
with some initial condition $\mb Y_0$ and random coefficients $M_n\in \mc M$ and $\bm \xi_n\in \Rd$, $ n\geq 1$. 
Indeed, by iteration, it is easy to see that the solution of \eqref{eq: opMD} can be rewritten as 
\begin{equation}\label{eq: opMDas_inf_sum}
    \mb X\overset{d}{=}\sum_{k=1}^\infty (U_1\cdot \ldots\cdot U_k)^Q\mb W_k,
\end{equation}
where $U_1, U_2, \ldots$ and $\mb W_1, \mb W_2, \dots$ are independent sequences of mutually independent random elements, sampled from the uniform distribution on $[0,1]$ and the measure $\nu$, respectively. Thus, the solution $\mb X$ of \eqref{eq: opMD} is a weak limit as $n\to \infty$ of $\mb Y_n$ given by \eqref{eq:randomAR} in the case of $M_k= U_k^Q$ and $\bm \xi_k=M_{k}\mb W_k$. 

The existence of a weak limit of $\mb Y_n$ as $n\to \infty$, independent of the initial condition $\mb Y_0$, under some mild conditions on $M_n$ and $\bm \xi_n$ was established in the celebrated work of Kesten \cite{kesten1973random}.
In particular, Kesten \cite[Thm. 6]{kesten1973random} showed that the limit exists if $-\infty <\mathbb{E}\log\|M_1\|<0$ and  $0<\mathbb {E}|\bm \xi_1|^\kappa<\infty$ for some $\kappa>0$. However, as shown in \cite{vervaat1979stochastic} for $d=1$, the last condition can be weakened to $\mathbb {E}\log^+ |\xi_1|<\infty$ under some assumptions on $M_1$, and this proof translates easily for $d>1$ in our case, using \eqref{eq: norm_exp_Q} below. Hereafter, we define \begin{equation*}
    \log^+x=\begin{cases}
        0, &0<x\leq 1\\
        \log x, &x>1.
    \end{cases}
\end{equation*}

The model \eqref{eq:randomAR} is also commonly referred to as the random-coefficient $AR(1)$ process and was proposed as a valid model in physical, biological and financial applications, see \cite{vervaat1979stochastic} and references therein. 
The one-dimensional version of this problem has been studied in a series of works, see e.g. \cite{ChamayouSchorr1975,grincevivcius1975limit,embrechts1994perpetuities}.

Another interpretation of the representation \eqref{eq: opMDas_inf_sum} can be obtained if we rewrite first 
\begin{equation}\label{eq: shot-noise}
    \mb X\overset{d}{=}\sum_{k=1}^{\infty} f(T_k)\mb W_k,
\end{equation}
where $T_k$ are the arrival times of the Poisson process with intensity $1$ and $f: \R_+\mapsto \mc M$ is the measurable function defined by
$f(t)=e^{-tQ}$. The sum in \eqref{eq: shot-noise} represents the weak limit of the sum
\begin{equation}\label{eq: shot-noise_finite_time}
    \sum_{k=1}^{\infty} f(t-T_k')\mathds{1}\{T_k'<t\}\mb W_k'
\end{equation}
as $t\to \infty$. Here, the sequences $T_k',\mb W_k'$, $k\geq 1$ have the same joint distribution as $T_k,\mb W_k$, $k\geq 1$.
The equation \eqref{eq: shot-noise_finite_time} can be interpreted as the accumulation, up to time $t$, of a multidimensional signal.
This signal appears at random time points according to the Poisson process with intensity $1$; the amplitudes are independent of each other and of the arrival times, and identically distributed according to the law $\nu$; the signal propagates in time according to the law $f(\cdot)$. For more details on such models, we refer to \cite{takacs1954secondary, takacs1955stochastic, westcott1976existence}.

\begin{remark}\label{rem:diff_Dickman}
    It is evident that the set of multidimensional Dickman distributions $\{\mbb D_{\theta, \sigma}, \theta>0, \sigma\in \mc S \}$ coincides with a subset $\mc D(\mathbb{I} , \mc S)$ of $\mc D(\mc M_+, H)$.
    Moreover, for $d=1$ the sphere $\mathbb S^0$ consists of two points $\mbb S ^0=\{-1,1\}$ and by choosing in this case $Q=\theta>0$ and $\sigma(\{1\})=1-\sigma(\{-1\})=1$ we can recover the Dickman distribution $GD_\theta$. 
    
    The one-dimensional distribution $\mc D(1/\theta,\nu)$ with a more general type of probability measure $\nu\in H$ appears naturally in the contexts of shot-noise processes and random difference equations as well as in the limit schemes, see, e.g, \cite{ChamayouSchorr1975, vervaat1979stochastic, pinsky18}. Notably, in the particular case $\nu(dx)=be^{-bx}dx$,$x>0$, with some $b>0$, the distribution $\mc D(1/\theta,\nu)$ coincides with the well-known Gamma distribution having characteristic function of the form
    \begin{equation*}
      \exp \left\{  \theta \int\limits_0^1 (e^{izu}-1) e^{-bu}\frac{du}{u}  \right\},~ z\in \mathbb{R}.
    \end{equation*}
\end{remark}

\begin{remark}
    Note that in the case of $\nu=\delta_{\mb 0}$ it holds that $\mb X=\mb 0$ almost surely, and this situation is trivial. We shall exclude it from consideration and, furthermore, assume that $\nu(\{\mb 0\})=0$.
\end{remark}

Note also that throughout the work, we will make use of the following estimate.
Let $Q\in \mc M_+$ be fixed, then there exist positive constants $c_1,c_2$ and $K_1,K_2$ such that for all $0\leq s\leq 1$ it holds that
\begin{equation}\label{eq: norm_exp_Q}
    c_1s^{K_1}|\mb x|\leq \|s^Q \mb x\|\leq c_2s^{K_2}|\mb x|,
\end{equation}
see the proof of \cite[Lem. 6.1]{urbanik1972levy}.

To establish properties of the operator Dickman distributions, we will need the following lemma.

\begin{lemma}\label{lem1}
    Let $Q\in\mc M_+$ be fixed.
    \begin{enumerate}[(i)]
        \item  \label{lem1:i} For all $\mb x\in \Rd$ the integral
    \begin{equation*}
        P_1(\mb x)=\int\limits_0^1( |s^Q\mb x|\wedge 1) \frac {ds} s 
    \end{equation*}
    is finite, and  $P_1(\mb x)\leq B_1\log(1+|\mb x|)$ for some $B_1>0$.
    \item  \label{lem1:ii} For all $\mb x\in \Rd$ the integral
    \begin{equation*}
        P_2(\mb x)=\int\limits_0^1( |s^Q\mb x|^2\wedge 1) \frac {ds} s 
    \end{equation*}
    is finite, and  $P_2(\mb x)\leq B_2\log(1+|\mb x|^2)$ for some $B_2>0$.
    \end{enumerate}
\end{lemma}

\begin{proof}
\eqref{lem1:i} First note that the function $f(t)=t \wedge 1$ is non-decreasing. We can assume, without loss of generality, that $c_2>1$ in \eqref{eq: norm_exp_Q}. Thus, we obtain
\begin{align*}
    \int\limits_0^1( |s^Q\mb x|\wedge 1) \frac {ds} s &\leq c_2
    \int\limits_0^1( s^{K_2}|\mb x|\wedge 1) \frac {ds} s= \frac {c_2} {K_2} \int\limits_0^1( s|\mb x|\wedge 1) \frac {ds} s\\
    &= \frac 
    {c_2} {K_2}  \int\limits_0^{|\mb x|} ( s\wedge 1) \frac {ds} s=
    \frac {c_2} {K_2}\left(1\wedge |\mb x|-\log\left (1\wedge |\mb x|^{-1}\right)\right),
\end{align*}
and the result follows.

\eqref{lem1:ii}  This is a direct consequence of \cite[Lem. 6.1]{urbanik1972levy}
\end{proof}
\begin{proposition}\label{prop: MD_class_cf}
    For any $Q\in \mc M_+$ and $\nu\in H$, the distribution  $\mc D =\mc D (Q,\nu)$ 
    has the characteristic function of the form
    \begin{equation}\label{eq:opMD_CF}
        \psi(\mb z)=\hat{\mc D}(\mb z)=\exp \left[\int\limits_{\R^d}\int\limits_0^1(e^{i\mb z\cdot s^Q\mb x}-1)\nu(d\mb x) \frac {ds} s  \right],\quad \mb z \in \R^d.
    \end{equation}

\end{proposition}

\begin{proof}
Note first that the finiteness of the integral in \eqref{eq:opMD_CF} follows by Lemma~\ref{lem1}\ref{lem1:i} and the assumption $\nu\in H$.
Notice now that \eqref{eq: opMD} yields
\begin{equation*}
    \psi(\mb z)=\int\limits_{0}^1\psi(u^{Q^*}\mb z)\hat{\nu}(u^{Q^*}\mb z)du,\quad \mb z \in \R^d.
    \end{equation*}
Replacing  the variable $\mb z$ with $r^{Q^*}\bom$ for $r\geq 0$ and $\bom\in \Rd$, where by convention $0^{Q^{*}}\bom=\mb 0$ for all $\bom\in\Rd$, we get
\begin{align*}
    \psi(r^{Q^*} \bom)=&\int\limits_{0}^1\psi((ur)^{Q^*}\bom )\hat{\nu}((ur)^{Q^*}\bom)du\\
    =&\frac 1 r \int\limits_{0}^r\psi(u^{Q^*}\bom )\hat{\nu}(u^{Q^*}\bom)du.
\end{align*}
Since the latter expression is differentiable with respect to $r$, we can obtain
  \begin{align*}
    \frac \partial {\partial r}\psi(r^{Q^*} \bom)
    =&- \frac 1 {r^2} \int\limits_{0}^r\psi(u^{Q^*}\bom )\hat{\nu} (u^{Q^*}\bom)du+\frac 1 r \psi(r^{Q^*}\bom )\hat{\nu} (r^{Q^*}\bom)\\
    =&\frac 1 r \psi(r^{Q^*}\bom )(\hat{\nu} (r^{Q^*}\bom)-1),
\end{align*}  
or, alternatively,
\begin{align*}
    \frac \partial {\partial r}\log \psi(r^{Q^*}\bom)
    =\frac 1 r (\hat{\nu} (r^{Q^*}\bom)-1).
\end{align*} 
It follows from the above that there exists $g: \Rd\mapsto \mathbb{C}$ such that
\begin{align*}
    \log \psi(r^{Q^*}\bom)
    =\int\limits_0^r(\hat{\nu} (s^{Q^*}\bom)-1) \frac {ds} s +g(\bom)=
    \int\limits_0^1(\hat{\nu} ((sr)^{Q^*}\bom)-1) \frac {ds} s +g(\bom).
\end{align*}
Letting $r=0$, we can deduce from $\psi(\mb 0)=1$ that $g(\bom)=0$ for all $\bom\in \Rd$.

Returning to the variable $\mb z$, we finally get
\begin{align*}
    \log \psi(\mb z )=
    \int\limits_0^1(\hat{\nu} (s^{Q^*}\mb z)-1) \frac {ds} s=\int\limits_{\Rd}\int\limits_0^1(e^{is^{Q^*}\mb z\cdot \mb x }-1) \nu(d \mb x) \frac {ds} s ,
\end{align*}
and the result follows since $\left(s^{Q^*}\right)^*=s^Q$.
\end{proof}

\begin{corollary}\label{cor: CF_opMD}
     The characteristic function $\psi(\mb z)$ of $\mc D (Q,\nu)$, $Q\in\mc M_+$, $\nu\in H$, can be written as
        \begin{equation*}
        \psi(\mb z)=\exp \left[\int\limits_0^1(\hat{\nu}(s^{Q^*}\mb z)-1)\frac {ds} s  \right].
    \end{equation*}
\end{corollary}

\begin{corollary}\label{cor: ID_opMD}
     The class $\mc D(\mc M_+, H)$ is a subset of $ID(\Rd)$ and for any~$Q\in \mc M_+$ and $\nu\in H$ the distribution $\mc D(Q,\nu)$ has the characteristic triplet $(\mb a, O, M)$, where 
     \begin{equation}\label{eq: opMD_LChar}
     \begin{gathered}
         \mb a=\int\limits_{\Rd}\int\limits_0^1s^Q\mb x \mathds{1}_{|s^Q\mb x|\leq 1}\nu(d\mb x) \frac {ds} s ,\\ M(B)=\int\limits_{\Rd}\int\limits_0^1 \mathds{1}_{B}(s^Q\mb x)\nu(d\mb x) \frac {ds} s ,\quad B\in \mc B(\R^d).
     \end{gathered}
     \end{equation}
\end{corollary}  

\begin{proof}
    The result follows directly from \eqref{eq:opMD_CF} after verification that $|\mb a|<\infty$ and $\int_{\Rd}(1\wedge |\mb y|^2)M(d\mb y)<\infty$. 
    Indeed, using Lemma~\ref{lem1}\ref{lem1:i} we obtain that
    \begin{equation*}
        |\mb a|\leq \int\limits_0^1\int\limits_{\R^d}(|s^Q\mb x|\wedge 1)\nu(d\mb x) \frac {ds} s \leq B_1\int\limits_{\R^d} \log(1+|\mb x|)\nu (d\mb x) < \infty.
    \end{equation*}
   Similarly, using Lemma~\ref{lem1}\ref{lem1:ii}, we get
     \begin{align*}
        \int\limits_{\R^d} (1\wedge |\mb y|^2)M(d\mb y)=
        & \int\limits_{\R^d} \int\limits_0^1 (1\wedge |s^Q\mb x|^2) \nu (d\mb x) \frac {ds} s\\
        \leq & B_2 \int\limits_{\R^d} \log(1+|\mb x|^2) \nu (d\mb x)<\infty.
    \end{align*}
    Note that we used the assumption $\nu\in H$ in both parts.
\end{proof}
\begin{remark}
    Remarkably, the distributions with L\'evy measure as in \eqref{eq: opMD_LChar} previously appeared in \cite{jurek1982structure}.
\end{remark}

\begin{proposition}\label{prop:mom_cov}
   \begin{enumerate}[(i)]
       \item \label{prop:mom_cov: i} Assume that $\nu\in H$ is such that $\int\limits_{\Rd}|\mb x|\nu(d\mb x) <\infty$. Then the mean vector $\mb m_{Q,\nu}$ of $\mc D(Q, \nu)$ exists for any $Q\in \mc M_+%
    $ and can be given by
    \begin{equation}\label{eq:mean_vec}
        \mb m_{Q,\nu}=\int\limits_{\R^d} \int\limits_0^1s^Q\mb x \nu(d \mb x)\frac {ds} s.
    \end{equation}
    \item  \label{prop:mom_cov: ii} Assume that $\nu\in H$ such that $\int\limits_{\Rd}|\mb x|^2\nu(d \mb x) <\infty$. Then the covariance matrix $\mathbb C_{Q,\nu}$ of $\mc D(Q, \nu)$ exists for any $Q\in \mc M_+%
    $  and can be given by
    \begin{equation}\label{eq:cov_matr}
        \mathbb C_{Q,\nu}=\int\limits_{\R^d} \int\limits_0^1 (s^Q\mb x) (s^Q\mb x)^T\nu(d \mb x) \frac {ds} s, 
    \end{equation}
   where $\mb x  \mb y^T=[x_iy_j]_{i,j=1}^d\in \mc M$ for any $\mb x,\mb y\in \Rd$.
   \end{enumerate}
\end{proposition}

\begin{proof}
    Both formulas \eqref{eq:mean_vec} and \eqref{eq:cov_matr} follow easily by differentiation of the characteristic function \eqref{eq:opMD_CF}. Note that the finiteness of the integrals in \eqref{eq:mean_vec} under the assumptions of \eqref{prop:mom_cov: i} follows again from \eqref{eq: norm_exp_Q}.
    Indeed, we can obtain
    \begin{equation*}
        |\mb m|\leq \frac {c_2} {K_2} \int\limits_{\Rd}|\mb x|\nu(d\mb x)<\infty. 
    \end{equation*}
    The finiteness of the integrals in \eqref{eq:cov_matr} follows similarly. 
\end{proof}

\begin{remark}
    Note that it follows from
    \begin{equation*}
        \frac d{dt} e^{tX}=Xe^{tX}=e^{tX}X,\quad t\in \R,~ X\in \mc {M},
    \end{equation*}
    and integration by parts, that formulas \eqref{eq:mean_vec} and \eqref{eq:cov_matr} can be rewritten as 
     \begin{equation*}
    \begin{gathered}
    \mb m_{Q,\nu}=Q^{-1}\int_{\Rd}\mb x \nu(d \mb x)\;\text{ and }\;
        \mathbb C_{Q,\nu}=\mc A_Q^{-1}\left(\int_{\Rd} \mb x \mb x^T \nu(d \mb x)\right),
    \end{gathered}
    \end{equation*}
    where $\mc A_Q: \mc M_d\mapsto \mc M_d$, $\mc A_Q(B)= QB+BQ^*$.
\end{remark}

\begin{proposition}\label{prop: conv}

    \begin{enumerate}[(i)]
        \item \label{prop: conv: i} For any fixed $Q\in \mc M_+$, the class $\mc D(Q, H)$ is closed under transformations $\mb x\mapsto t^A\mb x $ for any $t>0$ and $A\in \mc M$ commuting with Q (i.e. $AQ=QA$). 
        In particular, $\mc D(Q, H)$ is closed under transformations $\mb x\mapsto t^\alpha \mb x $ for any $t>0$ and $\alpha>0$.
        \item \label{prop: conv: ii} Consider the distribution $\mc D(\frac 1 \theta I, \nu)$, $\theta>0$, $\nu\in H$ and assume that the distribution $\nu$ is invariant under the action of a linear operator $T: \Rd\mapsto \Rd$, i.e. $T\nu=\nu$. Then $\mc D(\frac 1 \theta I, \nu)$ is also 
        invariant under  $T$. In particular, if $\nu$ is rotationally invariant,  
        $\mc D(\frac 1 \theta I, \nu)$ is also rotationally invariant for any $\theta>0$.
        
        \item  \label{prop: conv: iii} For any fixed $Q\in \mc M_+$, the class $\mc D(w^{-1}Q, H)$, where $w$ runs over $(0,\infty)$, is closed under finite convolutions.      
    \end{enumerate}
\end{proposition}
\begin{proof}
        \eqref{prop: conv: i} Let $Q\in \mc M_+$ be fixed and let $\mb X$ be a random vector such that $\mc L(\mb X)\in \mc D(Q, H)$. Then there exists a random vector $\mb W$, $\mc {L}(\mb W)\in H$, such that \eqref{eq: opMD} holds. Therefore, $t^A \mb X$ satisfies  
        \begin{equation*}
            t^A\mb X\overset{d}{=}t^AU^{Q}(\mb X+\mb W)
        \end{equation*}
        for any $t>0$.
        Since $A$ and $Q$ commute, then $s^A$ and $t^Q$ also commute for any $s, t>0$. Thus, we obtain  
        \begin{equation*}
            t^A\mb X\overset{d}{=}U^{Q}(t^A\mb X +t^A\mb W),
        \end{equation*}
        and, moreover,  $\mathbb{E} \log ^+ |t^A\mb W|\leq\mathbb{E}  \log^+ (\|t^A\|\cdot |\mb W|)<\infty$. This implies that $\mc L(t^A\mb X)\in \mc D(Q, H)$ or, more specifically, $t^A\mc D(Q, \nu)=\mc D (Q,t^A\nu)$.
        
        The proof of part \eqref{prop: conv: ii} is analogous.
        
        \eqref{prop: conv: iii} Note first that for fixed $Q\in \mc M_+$, $w>0$  and
         $\mc D=\mc D(w^{-1}Q,\nu)$, from \eqref{eq:opMD_CF} by a change of variables we obtain
          \begin{equation*}
             \log \widehat{\mc D}(\mb z)=\int\limits_{\R^d} \int\limits_0^1(e^{i\mb z\cdot s^{\frac 1 w Q}\mb x}-1)\nu(d\mb x) \frac {ds} s =
             w\int\limits_{\R^d} \int\limits_0^1(e^{i\mb z\cdot s^{Q}\mb x}-1)\nu(d\mb x) \frac {ds} s  .
         \end{equation*}
         Let $Q\in \mc M_+$, $w_1, w_2>0$ and $\nu_1,\nu_2\in H$ and consider  the independent  distributions $\mc D_1=\mc D(w_1^{-1}Q,\nu_1)$, $\mc D_2=\mc D(w_2^{-1}Q,\nu_2)$. If we denote $\mc D=\mc D_1*\mc D_2$, then
         \begin{align*}
            \log \widehat{\mc D}(\mb z)&=\int\limits_{\R^d} \int\limits_0^1(e^{i\mb z\cdot s^Q\mb x}-1) (w_1\nu_1+w_2\nu_2)(d\mb x)\frac {ds} s\\&= (w_1+w_2)\int\limits_{\R^d} \int\limits_0^1(e^{i\mb z\cdot s^Q\mb x}-1)\frac {(w_1\nu_1+w_2\nu_2)(d\mb x)} {w_1+w_2}\frac {ds} s,
         \end{align*}
         and it is easy to see that $\mc D=\mc D(\frac 1 {w_1+w_2}Q, \frac {w_1\nu_1+w_2\nu_2} {w_1+w_2})$.
    \end{proof}
    
\begin{remark}
    The case of countable convolutions is considered in Corollary \ref{cor: inf_conv}.
\end{remark}

Let us now show that each $\mc D\in \mc D(Q, H)$ is $Q$-selfdecomposable. The definition of operator selfdecomposability was first given in \cite{urbanik1972levy}.
 Following \cite{urbanik1972levy}, we say that a measure (distribution) on $\mathbb{R}^d$ is full if its support is not contained in any $(d-1)$-dimensional hyperplane.
 Urbanik proved that each full $Q$-selfdecomposable distribution appears as a weak limit of some triangular array $A_n\sum_{k=1}^n \mb X_k+\mb a_n,$  $n\in\mathbb{N}$, where $\mb X_n,$ $n\geq 1$ is a sequence of independent $\R^d$-valued random variables, $A_n\in\mc M_-$ and $\mb a_n\in \R^d,$ $n\geq 1$. The class of operator selfdecomposable distributions was studied in \cite{urbanik1972levy, sato1984operator, masuda2004multidimensional}, see also \cite{sato1985completely,meerschaert2001limit}  for the subclass of operator-stable distributions and applications in limit theorems.

Let us recall the definition of operator selfdecomposable distributions. 
\begin{definition}[\cite{urbanik1972levy}]\label{def:opSD}
    The distribution $\varrho\in \Pset(\R^d)$ is $Q$-selfdecomposable for some $Q\in\mc M_+$ if for every $t\in \R_+$ there exists $\varrho_t\in \Pset(\R^d)$ such that
\begin{equation*}
    \hat{\varrho}(\mb z)=\hat{\varrho}(e^{-tQ^*}\mb z)\hat{\varrho}_t(\mb z),\quad \mb z \in \R^d.
\end{equation*}
\end{definition}

In \cite{sato1984operator}, the authors presented two useful representations for the characteristic function of the operator selfdecomposable distribution, which allow us to establish the following proposition.
\begin{proposition}\label{prop:MD_opSD} 
    Let $Q\in \mc M_+$  and $\nu\in H$ be given, then $\mc D=\mc D(cQ,\nu)$, is $Q$-selfdecomposable for any $c>0$.
\end{proposition}
\begin{proof}
    It follows easily from \eqref{eq:opMD_CF} and representations given in \cite{sato1984operator}. Indeed, for each $\mc D(Q,\nu)$, we can choose $\rho=\nu$ in \cite[Second representation theorem of $Q$-selfdecomposable distributions]{sato1984operator}.
\end{proof}

Note that the classes of $Q$-selfdecomposable Dickman distributions $\mc D(cQ,H)$, $c>0$, might have non-empty intersections even for distinct $Q\in \mc M_+$, as proved in the proposition below. This question for the case of operator stable distributions was studied in \cite{hudson1981operator}.
\begin{proposition}\label{prop:SDMD_intersection}
    Consider $Q\in \mc M_+$ and assume that $Q$ has at least one real eigenvalue. Assume also $Q\notin \mathbb{I}$. Consider now $\mc D = \mc D(Q, \nu)$, with $\nu$ being the probability measure supported on the real eigenspaces of $Q$. Then also $\mc D\in \mc D(\mathbb{I}, H)$ and, consequently, $\mc D$ is both $Q$- and $I$-selfdecomposable.
\end{proposition} 
\begin{proof}
     Let $a_1,\ldots, a_n>0$, $n\geq 1$, denote the distinct real eigenvalues of $Q$ and $A_1, \ldots, A_n$ be the corresponding eigenspaces. Thus, we have $Q\mb x=a_k\mb x$ and  $s^Q\mb x=s^{a_k}\mb x$ for any $\mb x\in A_k$, $k=1,\ldots, n$. Denote also $p_k=\nu(A_k)$, $p_k\geq 0$, $k=1,\ldots, n$  and note that $p_1+\cdots +p_n=1$ by the assumptions above. It follows from  Proposition \ref{prop: MD_class_cf} that
    \begin{align*}
        \hat{\mc D}(\mb z)=&\exp \left[\int\limits_{\R^d} \int\limits_0^1(e^{i\mb z\cdot s^Q\mb x}-1)\nu(d\mb x) \frac {ds} s  \right]\\
        =&\exp \left[\sum_{k=1}^n\int\limits_{A_k} \int\limits_0^1(e^{i\mb z\cdot s^{a_k}\mb x}-1)\nu(d\mb x) \frac {ds} s  \right]\\
        =&\exp \left[\sum_{k=1}^n \int\limits_{A_k}\int\limits_0^1(e^{i\mb z\cdot s\mb x}-1) a_k^{-1 } \nu(d\mb x) \frac {ds} s \right].
    \end{align*}
    Let $\nu_k$ be the restriction of the measure $\nu$ to $A_k$, $k=1,\ldots, n$. Rewrite now
      \begin{align*}
         \hat{\mc D}(\mb z)=&\exp \left[ \int\limits_{\R^d} \int\limits_0^1(e^{i\mb z\cdot s\mb x}-1)  \sum_{k=1}^na_k^{-1 } \nu_k(d\mb x) \frac {ds} s \right]\\
         =&\exp \left[ \int\limits_{\R^d} \int\limits_0^1(e^{i\mb z\cdot s^\theta\mb x}-1)  \tilde{\nu}  (d\mb x)\frac {ds} s\right],
    \end{align*}
    where $\theta^{-1}=\sum_{k=1}^n a_k^{-1 } {p_k} $ and $\tilde{\nu}=\theta \sum_{k=1}^n  a_k^{-1 }{\nu_k} $ is the probability measure on $\Rd$. Consequently, we have $\mc D=\mc D(\theta I, \tilde{\nu})$.
\end{proof}

Let us now consider some specific examples of the distributions in the class of operator Dickman distributions.

\begin{example}\label{example: diag}
Let us fix $\Lambda\in \mc M_+$ and assume that $\Lambda$ is diagonal, that is $\Lambda=diag(a_1,\ldots, a_d)$ for $a_i>0, i=1,\ldots, d$. We will also assume that $\Lambda \notin \mathbb{I}$. Let $\mc L(\mb X)=\mc D(\Lambda, \nu)$ for some $\nu\in H$. Since in this case it holds for $s>0$ that $s^\Lambda=diag(s^{a_1},\ldots, s^{a_d})$, it is clear that the components $X_i$ of the vector $\mb X$ satisfy the distributional equation 
\begin{equation*}
    X_i\overset{d}{=}U^{a_i}(X_i+W_i),\quad i=1, \ldots, d,
\end{equation*}
where $W_i,$ $i=1, \ldots, d$ are components of the random element $\mb W$, $\mc L(\mb W)=\nu$. The random variables of this type have been considered in \cite{ChamayouSchorr1975,vervaat1979stochastic, pinsky18}.
Moreover, the formula \eqref{eq:cov_matr} for the covariance matrix in this case simplifies to
\begin{equation*}
[\mb C_{Q,\nu}]_{i,j}=\frac 1 {a_i+a_j}\int\limits_{\Rd} x_ix_j \nu(d \mb x) ,\quad i,j=1\ldots, d
\end{equation*}
whenever $\int_{\Rd}|\mb x|^2\nu(d\mb x) <\infty$.
Notice also that in the simplest case $\nu=\delta_\mb w$, for some $\mb w=(w_1,\ldots, w_d)\in \Rd_+$, the components $w_iX_i$ are distributed according to the Dickman distribution $GD_{1/a_i}$, $i=1,\ldots,d$.
\end{example}

\begin{example}\label{example: diagonalizable}
    Let us assume now that $Q\in \mc M_+$ can be represented as $Q=S\Lambda S^{-1}$ for $\Lambda$ as in Example \ref{example: diag} and some invertible matrix $S$. Then we have $t^Q=St^\Lambda S^{-1}$ for $t>0$. Let again $\mc L (\mb X)=\mc D(Q,\nu)$, $\nu \in H$, and $\mc L(\mb W)=\nu$. Then it follows from \eqref{eq: opMD}  that 
    \begin{equation*}
        \mb X\overset{d}{=}SU^\Lambda S^{-1}(\mb X+\mb W),
    \end{equation*}
    or, alternatively,
     \begin{equation*}
        S^{-1}\mb X\overset{d}{=}U^\Lambda (S^{-1}\mb X+S^{-1}\mb W),
    \end{equation*}
    and, therefore, $\mc L({S^{-1}\mb X})=\mc D(\Lambda, S^{-1}\nu)$, cf. Example \ref{example: diag}. 
\end{example}

\begin{example}\label{ex:S_meas}
Consider $\mc D=\mc D(\frac 1 \theta I, \nu)$, $\theta>0$, and assume $\nu \in \mc S$. As was noted before, the random vector $\mb X$ with $\mc L(\mb X)=\mc D$ has the multidimensional Dickman distribution as defined in \cite{bhattacharjee2020convergence}. The properties of this distribution have been studied in \cite{grabchak2024representation}, see also \cite{grabchak2024smalljumpslevyprocesses, grabchak2026simulation}.
In particular, it was established that if the measure $\nu$ is finitely supported, then $\mb X$ can be expressed using independent Dickman random variables. Indeed, 
assume that $\nu$ is of the form
    \begin{equation}\label{e:finitemeasure}
        \nu=\sum\limits_{i=1}^n p_i\delta_{\mb{w}_i},
    \end{equation}
where $\mb w_1, \ldots, \mb w_n\in \Sd$ and $p_1,\ldots, p_n>0$, $p_1+\cdots+ p_n=1$, $n\geq 1$. Then $\mb X$ can be represented as 
\begin{equation}\label{eq:finite_sum_XMD}
    \mb X\overset{d}{=} \sum\limits_{i=1}^n \mb w_i X_D^{(i)},
\end{equation}
where $X_D^{(1)}, \ldots, X_D^{(n)}$ are independent random variables and  $X_D^{(i)}$ is distributed according to $GD_{\theta p_i}$, $i=1, \ldots, n$.
The representation \eqref{eq:finite_sum_XMD} was obtained in \cite{grabchak2024representation}, and it stays valid for $\nu\in H$.
Moreover, the representation \eqref{eq:finite_sum_XMD} can be used for simulation from the multivariate Dickman distribution $\mc D(\frac 1 \theta I,\nu)$ as discussed in \cite{grabchak2024representation}. 

Recall also that the distribution $\mc D(Q,\nu)$ can be simplified to the distribution $\mc D( \theta I,\Tilde{\nu})$ with some $\theta>0$, $\Tilde{\nu}\in H$, in case probability measure $\nu$ is supported on real eigenspaces of the matrix $Q$, see Proposition \ref{prop:SDMD_intersection}.
\end{example}

Notice that if the operator selfdecomposable distribution over $\Rd$ is full, then it is absolutely continuous, see \cite{YAMAZATO1983550}.
Nevertheless, an explicit expression for the density is generally difficult to obtain.
In what follows, we focus on obtaining relations for the density function of the operator Dickman distributions in some particular cases.

In Theorem \ref{thm: unif_dens}, the explicit formula for the density is obtained for a specific case of the multidimensional Dickman distribution $\mc D\in \mc D(\mathbb I,\mc S)$, whereas Proposition \ref{prop: MD_diff_eq} treats a more general case of $\mc D(\mathbb I,H)$. Note that we utilise some elements of Fourier analysis and the theory of generalised functions operating on the Schwartz space of rapidly decreasing functions. For more details on these topics, we refer the reader to the classic books \cite{stein1971introduction}, \cite{gelfand1969generalized} and also \cite{mitrea2013distributions}. In this manner, we denote below $\Shw$, the space of Schwartz functions on $\Rd$ and $\Shw'$, its dual space, the space of tempered distributions. Moreover, we write $\overset{\Shw'}{\to}$ to indicate the convergence in the sense of $\Shw'$.

First, we consider the multidimensional Dickman distribution $\mc D(\frac 1 \theta I, \nu)$, where the distribution $\nu$ is uniform on the sphere $\Sd$. In the following theorem, we obtain its density function, which can be viewed as the multivariate analogue of the density $f_\theta(\cdot)$, see Section \ref{sec:Univ_D}. The idea of the proof is due to \cite[Thm. 4.7.7]{vervaat1972success}.

We further denote $J_a(\cdot), a>-\frac 1 2$, the Bessel function of first kind, $Leb$ the Lebesgue measure on $\Sd$,  $\zeta_d=(Leb(\Sd))^{-1}=\frac {\Gamma(d/2)} {2\pi^{d/2}}$, the reciprocal surface area of $\Sd$ and $Y_d(s)=( s/ 2)^{1-\frac d 2}\Gamma( d/ 2 )J_{ d/ 2-1}(s)$, $s>0$,  $d>1$, the spherical Bessel function, see e.g. \cite{pinsky1993fourier}. Moreover, for $d=1$, we let $Y_1(s)=\cos (s), s>0$.

\begin{theorem}\label{thm: unif_dens}
    Let $d\geq  1$, $\theta>0, \theta\neq d+2n, n=0,1,\ldots$, and consider the distribution $\mc D=\mc D(\frac 1 \theta I, \nu)$.
    Assume also that the measure $\nu$ is uniform on $\Sd$, i.e. $\nu=\zeta_d Leb$. If we let $\mc L(\mb X)=\mc D$, then the density $f(\mb x)$, $\mb x  \in \Rd$, of $\mb X$ exists and can be given by  
\begin{equation}\label{eq: unif_dens}
    f(\mb x)=e^{\theta\alpha_d}\left(g(\mb x)+\sum_{n=1}^\infty\frac {(-\theta \zeta_d)^n} {n!} \idotsint\limits_{D_n} g(\mb x-\mb x_1-\cdots -\mb x_n)\frac {d\mb x_1\ldots d\mb x_n}{ |\mb x_1|^d\ldots |\mb x_n|^d}\right),
\end{equation}
where 
\begin{gather}\label{eq:alpha_d}
    \alpha_d= \int\limits_0^1 \left(Y_d(s)-1\right)\frac {ds} s +\int\limits_1^{\infty}  Y_d(s)\frac {ds} s<\infty,\\
    g(\mb y)=  2^{-\theta} \pi^{-d/2}\frac {\Gamma(\frac{d-\theta} 2)} {\Gamma(\frac{\theta} 2)}|\mb y|^{\theta-d}, \quad \mb y\in \Rd,\nonumber\\
    D_n=\left\{(\mb x_1, \ldots, \mb x_n)\in \R^{d\times n}: |\mb x_i|>1,~ i=1, \ldots, n\right\} \nonumber.
\end{gather}
The formula \eqref{eq: unif_dens} is understood as a limit in the space of tempered distributions $\Shw'$ as given in \eqref{eq: f_lim_conv} below.

\end{theorem}

\begin{proof}
Note first that, according to the discussion above, the density $f(\mb x)$ exists and is well-defined almost everywhere on $\Rd$. Thus, we can identify 
$\hat {\mc D}(\mb z)=\hat{f}(\mb z)$, $\mb z\in \Rd$.

Let us consider the function $\varphi(\mb z)=\log \hat{f}(\mb z)$. Using Corollary \ref{cor: CF_opMD}, we obtain
\begin{equation*}
     \varphi(\mb z)=\theta \int_0^1(\hat{\nu}(s\mb z)-1)\frac {ds} s.
\end{equation*} 
Moreover, for $\nu=\zeta_d Leb(B)$ we have 
\begin{equation*}
     \hat{\nu}(\mb w)= Y_d(|\mb w|),\quad \mb w\in \Rd,
\end{equation*}
see, for example, \cite{pinsky1993fourier}.
Therefore, we obtain for $\mb z\in \R^d\setminus\{\mb 0\}$ that
\begin{align*}
    \theta^{-1}\varphi(\mb z)&= \int\limits_0^{1} \left(Y_d(s |\mb z |)-1 \right ) \frac {ds} s=
     \int\limits_0^{|\mb z |} \left( Y_d(s) -1 \right ) \frac {ds} s \\
    &=\int\limits_0^1\left(  Y_d(s) -1 \right )\frac {ds} s +\int\limits_1^{\infty} Y_d(s) \frac {ds} {s} - \int\limits_{|\mb z |}^\infty Y_d(s) \frac {ds} {s}-\log|\mb z | .
\end{align*}
Since the function $Y_d(\cdot)$ satisfies
\begin{equation}\label{Y_d_as}
    \begin{gathered}
         Y_d(s)=1+O(s^2), \; s\to 0;\\
    Y_d(s)=\left(\frac s 2 \right)^{\frac {1-d}2 }\Gamma\left(\frac d 2\right)\pi^{-1/2}\left[\cos\left(s-\frac {\pi(d-1)} 4\right)+O\left(\frac 1 s\right)\right], \; s\to\infty,
    \end{gathered}
\end{equation}
we get \eqref{eq:alpha_d}; see \cite[Prop. 2.1]{pinsky1993fourier}.

Consequently, we can rewrite
\begin{align}
    \hat{f}(\mb z)=&\exp\left(\theta\alpha_d+\log|\mb z |^{-\theta}-\theta  \int\limits_{1}^\infty  Y_d(s|\mb z |)\frac{ds} s \right)\nonumber\\
    =&e^{\theta\alpha_d} |\mb z|^{-\theta} \exp\left(-\theta \int\limits_{1}^\infty  Y_d(s|\mb z |)\frac{ds} s \right)\label{eq: fhat_repr}.
\end{align}

We note first that for $0<\theta<d$, the function $|\mb z|^{-\theta}$, $\mb z\in \Rd\setminus\{\mb 0\}$, defines a tempered distribution, which admits analytical continuation for all $\theta>0$, $\theta\neq d+2n$, $n=0,1,\ldots$; see \cite[Ch.~I]{gelfand1969generalized}. Additionally, its inverse Fourier transform in the sense of $\Shw'$ is given by the tempered distribution $g$ defined above, as shown in \cite[p.~194]{gelfand1969generalized}.

Denote now $\Tilde{F}_k(s)=s^{-d}\mathds{1} \{1<s<k\}$,  $k\geq2$,  and $\Tilde{F}(s)=s^{-d}\mathds{1} \{s>1\}$, $s\in \R$. Let also $F_k(\mb x)=\Tilde{F}_k(|\mb x|)$ and $F(\mb x)=\Tilde{F}(|\mb x|)$, $\mb x \in \R^d$. Note that $\{F_k, k\geq 2\}$ is a sequence of functions such that 
\begin{enumerate}
    \item $F_k(\mb x) \in L_1(\Rd)\cap L_2(\Rd)$,
    \item $F_k(\mb x) \overset{L_2(\Rd)}{\to}F(\mb x)$ as $k\to \infty$,
    \item $F_k$ is radial for any $k\geq 2$ and 
    \begin{align*}
        &\int_{\Rd}e^{i\mb z\cdot \mb x}F_k(\mb x)d\mb x =(2\pi)^{d/2}|\mb z|^{1-d/2 }\int_0^\infty \Tilde{F}_k(s)J_{d/2-1}(s|\mb z|)s^{d/2}ds\\
        =& \zeta_d^{-1}\int_0^\infty \Tilde{F}_k(s)Y_d(s|\mb z|)s^{d-1}ds= \zeta_d^{-1}\int_1^k Y_d(s|\mb z|)\frac {ds} s, \quad \mb z\in \Rd, 
    \end{align*}
    see \cite{stein1971introduction}. Note that for $d=1$ this relation is obvious due to properties of even functions.
\end{enumerate}
Thus, the function $-\theta \int\nolimits_{1}^\infty  Y_d(s|\mb z |)\frac{ds} s$ is the Fourier transform of the function $-\theta\zeta_d F(\mb x)\in L_2(\Rd)$; see \cite{stein1971introduction} for more details on the theory of the Fourier transform on $L_2(\Rd)$.

Subsequently, we denote for $k\geq 2$ and $\mb z\in \Rd$
\begin{equation*}
    H_k(\mb z)=\exp\left(-\theta \int_{1}^k  Y_d(s|\mb z |)\frac{ds} s \right)=\exp\left(-\theta \zeta_d \int_{\Rd} e^{i\mb z\cdot \mb x} |\mb x|^{-d}  \mathds 1_{(1,k)}(|\mb x|)d\mb x \right).
\end{equation*}
We first note that from \eqref{Y_d_as} it is easy to see that $H_k(\mb z)\to H(\mb z)$ as $k\to\infty$ for all $\mb z\in \Rd\setminus\{\mb 0\}$. Additionally, 
for $k$ large enough and $\mb z\in \Rd$ away from $\{\mb 0\}$, it holds that
\begin{equation*}
    \left| |\mb z|^{-\theta}H(\mb z)-|\mb z|^{-\theta}H_k(\mb z) \right|=
    |\mb z|^{-\theta}H(\mb z) \cdot   \left| 1-e^{\theta \int_{k}^\infty  Y_d(s|\mb z |)\frac{ds} s }\right|\leq C_1 \hat f(\mb z),
\end{equation*}
where $C_1$ is some positive constant, and the last inequality follows again by \eqref{Y_d_as}.
Whereas, since functions  $H(\mb z)$ and $H_k(\mb z)$, $k\geq 2$, are uniformly bounded for small $|\mb z|$, for $k\geq 2$ and $\mb z$ in the neighbourhood of $\{\mb 0\}$ it holds that 
\begin{equation*}
     \left| |\mb z|^{-\theta}H(\mb z)-|\mb z|^{-\theta}H_k(\mb z) \right|\leq C_1' |\mb z|^{-\theta},
\end{equation*}
where $C_1'>0$. 

We can then conclude, by the dominated convergence theorem, that
\begin{equation}\label{eq: H_k_to_H}
    |\mb z|^{-\theta}H_k(\mb z)\overset{\Shw'}{\to}|\mb z|^{-\theta}H(\mb z), \quad k\to \infty.
\end{equation} 

Note also that for any $k\geq 2$, the function $\int_{1}^k  Y_d(s|\mb z |)\frac{ds} s=\zeta_d \int_{\Rd} e^{i\mb z\cdot \mb x} F_k(\mb x)d\mb x$ is uniformly bounded for $\mb z\in \Rd$ as can be derived from \cite[Eq. (3.1.2)]{mitrea2013distributions}. It then follows that 
\begin{align*}
   H_k(\mb z)=&\lim_{N\to\infty}\sum_{n=0}^N\frac {(-\theta\zeta_d)^n} {n!} \left(\int_{\Rd} e^{i\mb z\cdot \mb x}F_k(\mb x)d\mb x\right)^n\\
   =&\lim_{N\to\infty}\sum_{n=0}^N\frac {(-\theta\zeta_d)^n} {n!} \int_{D_{n,k}} e^{i\mb z\cdot (\mb x_1+\cdots+\mb x_n)}\frac {d\mb x_1\cdots d\mb x_n} { |\mb x_1|^d\ldots |\mb x_n|^d} \\
   =&\lim_{N\to\infty}\int_{\Rd} e^{i\mb z\cdot\mb x }\left( \sum_{n=0}^N\frac {(-\theta\zeta_d)^n} {n!} F_k^{(*n)}(\mb x) \right)d\mb x\eqqcolon \lim_{N\to\infty} H_{k,N}(\mb z),
\end{align*}
where $D_{n,k}=\left\{(\mb x_1, \ldots, \mb x_n)\in \R^{d\times n}: 1<|\mb x_i|<k,~ i=1, \ldots, n\right\}$, $k\geq 2$, $n\in \mathbb N$, and $^{(*n)}$ denotes the $n$-fold convolution, $F^{(*0)}(\cdot)\coloneqq\delta_{\mb 0}(\cdot)$. The convergence above is uniform for $\mb z\in \Rd$.
Consequently, there exists $N$ large enough such that, for some constant $C_2>0$ it holds that
\begin{equation*}
     \left| |\mb z|^{-\theta}H_k(\mb z)-|\mb z|^{-\theta}H_{k,N}(\mb z) \right|\leq C_2|\mb z|^{-\theta}, \quad |\mb z|\in \Rd\setminus\{\mb 0\}.
\end{equation*}
Hence, by the same reasoning as above, 
$|\mb z|^{-\theta}H_{k,N}(\mb z)\overset{\Shw'}{\to}|\mb z|^{-\theta}H_{k}(\mb z)$, $N\to\infty$. Combining it with \eqref{eq: H_k_to_H}, we deduce that
\begin{equation}\label{eq: H_kN_toH}
    \lim_{k\to \infty}\lim_{N\to \infty}|\mb z|^{-\theta}H_{k,N}(\mb z)=|\mb z|^{-\theta}H(\mb z),
\end{equation}
where the limit is understood in the sense of $\Shw'$.

Note now that it follows from above that for each $k\geq 2$ and $N\geq 0$, the tempered distribution $H_{k,N}$ is the Fourier transform of a distribution $h_{k,N}$ given by 
\begin{equation*}
    h_{k,N}(\mb x) =\sum_{n=0}^N\frac {(-\theta\zeta_d)^n} {n!} F_k^{(*n)}(\mb x), \quad \mb x\in \Rd.
\end{equation*}
Moreover, the tempered distribution $h_{k,N}(\mb x)$ is compactly supported for each $k\geq 2$ and $N\geq 0$; see \cite[Ch. 2]{mitrea2013distributions} for the definition and further properties.

Thus, the distribution $|\mb z|^{-\theta}H_{k,N}(\mb z)$ can be written as the Fourier transform on the space $\Shw'$ of the convolution \begin{equation*}
     (g*h_{k,N})(\mb x)=g(\mb x)+\sum_{n=0}^N\frac {(-\theta\zeta_d)^n} {n!} \int_{\Rd}g(\mb y)F_k^{(*n)}(\mb x-\mb y)d\mb y,\quad \mb x\in \Rd\setminus\{\mb 0\},
\end{equation*}
see \cite[Thm. 4.33]{mitrea2013distributions}.
Moreover, by the continuity of the Fourier transform on the space $\Shw'$, combined with \eqref{eq: H_kN_toH} and \eqref{eq: fhat_repr}, it follows that 
\begin{equation}\label{eq: f_lim_conv}
    f(\mb x)=\lim_{k\to \infty}\lim_{N\to \infty} e^{\theta \alpha_d} (g*h_{k,N})(\mb x), \quad \mb x\in \Rd,
\end{equation}
where the limit above is in $\Shw'$ sense. This completes the proof and gives a rigorous meaning to the formula \eqref{eq: unif_dens}.
\end{proof}

\begin{remark}
    Note also that for $d=1$, relation  \eqref{eq: unif_dens} yields a formula for the convolution $f_{\theta/2}(x)*f_{\theta/2}(-x)$, $\theta>0, \theta\neq 1+2n$, $n\geq 0$, where $f_{\theta/2}(x)$ is the density of the one-dimensional Dickman distribution $GD_{\theta/2}$. Moreover, it easily follows from \cite[Formula 3.782(1)]{gradshteyn2014table} that $\alpha_1=\gamma$, cf. Section \ref{sec:Univ_D}. 

    Notice also that the restriction $\theta\neq d+2n$, $n\geq 0$, arises from the analytical techniques used in the proofs, rather than from any structural property of the distribution $\mc D(\frac 1 \theta I,\nu)$.  In fact, the density $f(\mb x)$ of $\mc D(\frac 1 \theta I,\nu)$ exists for any $\theta>0$, and the expression \eqref{eq: unif_dens} remains valid except for values $\theta= d+2n$, $n\geq  0$.
\end{remark}

Observe also that from \eqref{eq: fhat_repr} combined with \cite[Prop. 28.1]{sato1999levy}, we obtain the following corollary.
\begin{corollary}
    Assume that the assumptions of Theorem \ref{thm: unif_dens} are satisfied for $d\geq 1$ and $\theta>0$. Then the characteristic function $\hat f(\mb z)$ of the operator Dickman distribution $\mc D (\frac 1 \theta I, \nu)$ satisfies 
    \begin{equation*}
        \hat f(\mb z)\sim e^{\theta \alpha_d} |\mb z|^{-\theta}, \quad |\mb z|\to \infty.
    \end{equation*}
    This, in turn, implies that all partial derivatives of $f(\mb x)$, $\mb x\in \Rd$, up to order $n\geq 1$, exist and are continuous whenever $\theta>d+n$.
\end{corollary}

The following proposition establishes the integro-partial differential equation for the density of distributions in $\mc D(\mathbb{I}, H)$. It thus gives a multivariate counterpart for the equation \eqref{eq: Dickman_ODE}, see also \cite[Lem. 3]{ChamayouSchorr1975}. Our result, however,  holds in a weak sense, treating the density as a tempered distribution. See \cite{mitrea2013distributions} for a definition of weak derivative.

\begin{proposition}\label{prop: MD_diff_eq}
Let $\theta>0$ and $\nu\in H$ be given. Let $f(\mb x)$, $\mb x\in \Rd$, denote the probability density function (given as a generalised function) of the $\mc D(\frac 1 \theta I, \nu)$  distribution. Then it satisfies the integro-partial differential equation 
\begin{equation*}   
\sum_{k=1}^d x_k\frac { \partial f(\mb x)} {\partial x_k} =(\theta-d)f(\mb x)-\theta\int\limits_{\Rd}f(\mb x-\mb y)\nu(d\mb y)
\end{equation*}
in a weak sense.
\end{proposition}
\begin{proof}
Notice that, due to Proposition \ref{prop: MD_class_cf}, the Fourier transform $\hat f $ of $f$ satisfies 
\begin{equation*}
    \log \hat{f}(\mb z) =\theta \int\limits_{\R^d} \int\limits_0^1(e^{is\mb z \cdot \mb x}-1)\nu(d\mb x) \frac {ds} s ,\quad \mb z\in \Rd.
\end{equation*}
Consequently, 
\begin{equation*}
    \sum_{k=1}^d z_k\frac { \partial \hat{f}(\mb z)} {\partial z_k} =\theta\hat{f}(\mb z)\hat{\nu}(\mb z)-\theta\hat{f}(\mb z).
\end{equation*}
The result now follows from the properties of the Fourier transform on $\Shw'$ and the convolution property of probability measures; see \cite[Thm. 4.25]{mitrea2013distributions} and \cite[Prop. 2.5 (iii)]{sato1999levy}.
\end{proof}
 \begin{remark}
    The differential operator $L=\sum_{k=1}^d z_k\frac{\partial}  {\partial z_k}$ is sometimes referred to as the Euler partial differential operator.
\end{remark}

\section{Some applications of the operator Dickman distributions}\label{sec:appl}

The one-dimensional Dickman distribution has applications in various probabilistic problems and beyond, as discussed in Section \ref{sec:Univ_D}. Moreover, the multidimensional Dickman distribution from Definition \ref{def:MD} was proven to be useful in approximating small jumps of some multidimensional L\'evy processes. This section presents examples of how the above-mentioned applications can be extended to the case of operator Dickman distributions.

 Hereafter ``$\overset{v}{\to}$'' and ``$\overset{w}{\to}$'' denote respectively vague and weak convergence of measures. In particular,  $\mu_n \overset{v}{\to} \mu$, $n\to\infty$, for measures $\mu_n$ and $\mu$ over $\Rd$, whenever
\begin{equation*}
 \int\limits_{\Rd} f(\mb x)\mu_n (d\mb x)\to \int\limits_{\Rd} f(\mb x)\mu (d\mb x),\quad n\to\infty,
\end{equation*}
for all bounded, continuous maps $f:\Rd \mapsto \R$ vanishing on a neighbourhood of $\mb 0$ \cite[p. 41]{sato1999levy}.
Moreover, by convention, ``$\overset{w}{\to}$'', when applied to random elements, means the weak convergence of their corresponding probability measures.

First, we present the following multidimensional counterpart for \cite[Ex. 4.7.5]{vervaat1972success}.
\begin{theorem}
Consider a sequence $\varepsilon_1, \varepsilon_2,\ldots $ of independent random variables such that  $\mbb P(\varepsilon_k=1)=1-\mbb P(\varepsilon_k=0)=p_k$, $k\geq 1$, and assume that
\begin{gather*}
    p_k\in [0,1), \quad k\geq 1,\\
    \lim\limits_{k\to\infty} p_k=0\quad \text{and}\quad  \sum_{k=1}^\infty p_k=\infty.
\end{gather*}
    Let $L(n)$ denote the index of $n$-th $1$ in the sequence $\varepsilon_k$, $k\geq 1$.
    
    Fix now $Q\in \mc M_{+}$ and assume that $Q$ is as in the Example \ref{example: diagonalizable}. Then 
    \begin{equation*}
        (L(n))^Q\sum_{k=1}^\infty (L(n+k))^{-Q}\mb W_k\overset{w}{\to} \mb X, \quad n\to \infty,
    \end{equation*}
    where $\mb W_1, \mb W_2, \dots$ is the sequence of independent identically distributed random elements with a common law $\nu\in H$, independent also of a sequence $\varepsilon_1, \varepsilon_2,\ldots $, and $\mc L(\mb X)=\mc D(Q, \nu)$.
\end{theorem}
\begin{proof}
    By the assumption $Q=S\Lambda S^{-1}$ for some invertible matrix $S$ and a diagonal matrix $\Lambda=diag\{a_1,\ldots,a_d\}$, $a_1,\ldots,a_d>0$. Then we have $t^Q=S \cdot diag\{t^{a_1,},\ldots,t^{a_d}\}\cdot S^{-1}$ for any $t>0$. Since $a_j>0$, $j=1,\ldots, d$, it follows from \cite[Ex. 4.7.5]{vervaat1972success} that 
    \begin{equation*}
        (L(n))^{a_j}\sum_{k=1}^\infty (L(n+k))^{-a_j}\overset{w}{\to} \sum_{k=1}^\infty (U_1\cdot \ldots\cdot U_k)^{a_j} \text{ as } n\to \infty,
    \end{equation*}
    for any $j=1,\ldots, d$.
    Rewrite now
    \begin{equation*}
        (L(n))^Q\sum_{k=1}^\infty (L(n+k))^{-Q}\mb W_k= S (L(n))^{\Lambda} \sum_{k=1}^\infty (L(n+k))^{-\Lambda} S^{-1}\mb W_k.
    \end{equation*}
    The result follows then from the continuous mapping theorem \cite[Thm. 2.7]{billingsleyconvergence},  \cite[Thm. 2.8]{billingsleyconvergence} and \eqref{eq: opMDas_inf_sum}.
\end{proof}

\begin{example}
    Notice that the sequence $L(n), n\geq 1$, is an integer-valued Markov chain with transition probability function given by 
    \begin{equation*}
        P(m,k)=\mathbb P(L(n)=k|L(n-1)=m)=p_k\prod_{j=m+1}^{k-1} (1-p_j),\quad k\geq m+1\geq 2,
    \end{equation*}
    where $n\geq 1$ and $L(0)\coloneq0$.
    Let us consider, for example, $p_k=\alpha/k$, $k\geq 1$, for some $\alpha\in (0,1)$. It then holds that
    \begin{align*}
        \mathbb P(L(n)-L(n-1)&=r |L(n-1)=m)\\
        &=\left(1-\frac \alpha {m+1}\right)\cdot \ldots \cdot \left(1-\frac \alpha {m+r-1}\right)\frac \alpha {m+r},\quad r\geq 1
    \end{align*}    
    and we can recognise that, conditionally on $L(n-1)=m$, the random variable $L(n)-L(n-1)$, $n\geq 1$, has generalised Sibuya distribution $S_1(\alpha, m)$ as defined in \cite{kozubowski2018generalized}.
    
    The sequence $L(n)$ has been studied extensively in relation to record epochs in \cite{vervaat1972success}.
\end{example}

Recall now that it has been proven in \cite{covo2009approximations}
that the Dickman distribution can be used to approximate small jumps of L\'evy processes when the Brownian approximation is not applicable. It therefore makes simulation of the L\'evy process (or the corresponding infinitely divisible random variable) easier. The analogous approximation for some class of multidimensional  L\'evy processes has been established in \cite{grabchak2024smalljumpslevyprocesses} using the multidimensional Dickman distribution defined as in \eqref{eq: MD}. We show here that these results can be extended to a wider class of infinitely divisible distributions using the operator Dickman distributions $\mc D(Q, \sigma)$ with $Q\in\mc M_+$ and $\sigma\in \mc S$. Note that our aim here is not to establish the most general case, but rather to demonstrate the possible scope of applications of the newly defined distributions.

We denote $\bar{\R}^d$, a one-point compactification of $\Rd$ and $\bar{\R}^d_0=\bar{\R}^d\setminus\{0\}$, $\bar{\R}_+=(0,\infty]$. We also introduce for any $Q\in \mc M_+$ the set $S_Q$ given by
\begin{equation*}
    S_Q=\{\mb x\in \Sd: |r^Q\mb x|>1 \text{ for every } r>1\},
\end{equation*}
and a homeomorphic map
\begin{equation*}
    \xi_Q: \bar{\R}_+\times S_Q \mapsto \bar{\R}^d_0, \quad \xi_Q(r,\mb x)=r^Q\mb x,
\end{equation*}
see the proof of \cite[Prop. 2]{jurek1982structure}.

To present the main result, we need the following generalisation of \cite[Lemma 4.9]{grabchak2015tempered}.

\begin{lemma}\label{lem: vague_conv}
    Let $\mu, \mu_1,\mu_2, \ldots $ be Borel measures on $\bar{\R}^d_0$. Assume also that $\mu$ is of the form 
    \begin{equation}\label{eq: lemma_jumps}
        \mu(B)=\int_{\Sd}\int_{\bar{\R}_+}\mathds{1}_B(r^Q\mb x)\sigma(d\mb x) \lambda(dr),\quad  B\in \mc B(\bar{\R}^d_0),
    \end{equation}
    where $\sigma$ is a probability measure on $\Sd$ and $\lambda$ is a Borel measure on $\bar{\R}_+$ such that $\lambda(\{a\})=0$ for any $0<a<\infty$ and $\lambda$ is finite on any set bounded away from $0$.   
    Then 
    \begin{equation*}
        \mu_n\overset{v}{\to}\mu \quad \text{if and only if}\quad \mu_n(A_{t,D})\to\mu(A_{t,D}),\quad n\to \infty,
    \end{equation*}
    for all $A_{t,D}=\xi_Q((t,\infty]\times D)=\{r^Q\mb x: r>t, \mb x\in D\}$ with $t>0$ and $ D\in \mc B(S_Q)$ such that $\sigma(\partial D)=0$.
\end{lemma}

\begin{proof}
    The proof repeats the proof of \cite[Lem. 4.9]{grabchak2015tempered} almost word by word due to continuity and bijectivity of the map $\xi_Q$.
\end{proof}

Recall now that for $ Q\in \mc M_+$ and $\sigma\in\mc S$, the L\'evy measure of the distribution $\mc D(Q,\sigma)$ is given by 
\begin{equation}\label{eq: Lmeas_D_S}
    M(B)=\int_{\Sd} \int_0^1 \mathds{1}_{B}(s^Q\mb x) \sigma(d\mb x) \frac {ds} s,\quad B\in \mc B(\R^d),
\end{equation}
which is of the form \eqref{eq: lemma_jumps}, see Corollary \ref{cor: ID_opMD}. 

We will further limit our attention to $ID(\Rd)$ distributions $\varrho$ with characteristic triple $(\mb a, O, \mu)$ such that the L\'evy measure $\mu$ satisfies the additional condition 
\begin{equation}\label{eq:lam-cond}
\int_{|\mb x|<1}|\mb x |\mu(d\mb x)<\infty.
\end{equation}
Moreover, choosing for simplicity $\mb a =\int_{|\mb x|<1}\mb x \mu (d\mb x)$, we obtain
 \begin{equation}\label{eq: mod_ID_CF}
    \hat{\varrho}(\mb z)=\exp\left[\int_{\Rd}(e^{i\mb z \cdot \mb x}-1)\mu(d\mb x) \right].
    \end{equation} 
    
\begin{theorem}\label{thm: convergence}
    Fix $Q\in \mc M_+$ and $\sigma\in \mc S$. Let $\varrho_1, \varrho_2, \ldots\in ID(\Rd)$ be a sequence of infinitely divisible distributions and let  $\mu_1, \mu_2, \ldots$ denote the sequence of their L\'evy measures, which we assume satisfy the  condition \eqref{eq:lam-cond}.
    Assume additionally that 
    \begin{equation}\label{eq: lem_converg_tightness_cond}
        \lim_{\delta \to 0} \limsup_{n \to \infty} \int_{|\mb x|<\delta}|\mb x|\mu_n(d\mb x)=0.
    \end{equation}
    Then $\varrho_n\overset{w}{\to}\mc D(Q,\sigma)$ as $n\to \infty$ whenever it holds for every $t>0$ and every $D\in \mc B(S_Q)$, $\sigma(\partial D)=0$, that
    \begin{equation}\label{eq: lem_converg_cond}
        \mu_n(A_{t,D})\to -\log (1\wedge t)\sigma (D),\quad n\to\infty,
    \end{equation}
    where $A_{t,D}$ is given as in Lemma \ref{lem: vague_conv}.
\end{theorem}

\begin{proof}
    Due to \cite[Thm. 8.7]{sato1999levy}, it is enough to check the following three conditions
    \begin{enumerate}[(i)]
        \item \label{eq:sato:i} $\mu_n\overset{v}{\to}M$;
        \item \label{eq:sato:ii} $\lim\limits_{\delta\to 0} \limsup\limits_{n\to \infty}  \int\limits_{|\mb x|\leq \delta} (\mb z\cdot \mb x)^2\mu_n(d\mb x)=0$  for any $\mb z\in \Rd$;
        \item \label{eq:sato:iii} $\lim\limits_{n\to \infty}\int\limits_{\Rd} \mb x c(\mb x) \mu_n(d\mb x)=\int\limits_{\Rd} \mb xc(\mb x) M(d\mb x)$ for some bounded continuous function  $c: \Rd\mapsto \R$ satisfying conditions of \cite[Thm. 8.7]{sato1999levy}.
    \end{enumerate}
    It follows from Lemma \ref{lem: vague_conv} and \eqref{eq: Lmeas_D_S} that conditions \eqref{eq:sato:i} and \eqref{eq: lem_converg_cond} are equivalent in the case of the $\mc D (Q,\sigma)$ distribution.
    Notice also that the condition \eqref{eq:sato:ii} follows readily from \eqref{eq: lem_converg_tightness_cond}. Consider condition \eqref{eq:sato:iii} now. Assuming that \eqref{eq:sato:i} holds, we can rewrite for any $\delta>0$
    \begin{align*}
        \lim\limits_{n\to \infty}\int\limits_{\Rd} \mb x c(\mb x) \mu_n(d\mb x)=&\lim\limits_{n\to \infty} \left( \int\limits_{|\mb x|<\delta} \mb x c(\mb x) \mu_n(d\mb x)+\int\limits_{|\mb x|\geq \delta} \mb x c(\mb x) \mu_n(d\mb x)\right)\\
        =&\lim\limits_{n\to \infty}  \int\limits_{|\mb x|<\delta} \mb x c(\mb x) \mu_n(d\mb x)+ \int\limits_{|\mb x|\geq \delta} \mb x c(\mb x) M(d\mb x),
    \end{align*}
    where the last equality follows from the of vague convergence $\mu_n\overset{v}{\to}M$ since the maps $\Rd\mapsto \R$, $\mb x\mapsto x_k c(\mb x)$, $k=1,\ldots, d$, define bounded continuous functions; see \cite[Thm. 1]{barczy2006portmanteau}. We can now let $\delta\to 0$ and \eqref{eq:sato:iii} follows from \eqref{eq: lem_converg_tightness_cond}, see also  \cite[Thm. 5.29]{kallenberg2021foundations}.
\end{proof}

We now consider the specific example of a sequence satisfying the conditions of Theorem \ref{thm: convergence}. Namely let $\mu$  be a L\'evy measure on $\Rd$ such that
\begin{equation}\label{eq: ID_mu}
    \mu(B)=\int_{\Sd} \int_0^\infty\mathds{1}_B\{r^Q\mb x\} \varpi(d\mb x, dr), \quad B\in \mc B(\Rd),
\end{equation}
where $Q\in \mc M_+$ is fixed and $\varpi(\cdot, \cdot)$ is a measure on $\Sd\times \R_+$ such that 
\begin{equation}\label{eq: mu_Id_Levycond}
    \int_{\Sd} \int_0^\infty|r^Q\mb x| \mathds{1}_{B_1}(r^Q\mb x) \varpi(d\mb x, dr)<\infty,
\end{equation}
where we denote $B_h\subset \Rd$, an open ball of radius $h>0$.
Thus, measure $\mu$ satisfies \eqref{eq:lam-cond} and there exists $\rho\in ID(\Rd)$ such that \eqref{eq: mod_ID_CF} holds.

For the measure $\mu$ and some $u>0$, we consider the truncation  
    \begin{equation*}
         T_u\mu(B)=\int\limits_{\Sd} \int\limits_0^{u}\mathds{1}_B\{r^Q\mb x\}   \varpi(d\mb x, dr),\quad B\in \mc B(\Rd),
    \end{equation*}
and note that $T_u\mu$ is a L\'evy measure for any $u>0$. Note also that the truncation $T_u$ corresponds to considering only small jumps of the underlying Poisson random measure, whereas the measure $\mu-T_u\mu$ corresponds to the compound Poisson part.
   Denote then $T_u\varrho$, the $ID(\Rd)$ distribution with L\'evy measure  $T_u\mu$. Consider further the transformation $\varrho \mapsto \varepsilon^{-Q} T_\varepsilon\varrho$, $\varepsilon>0$, which accords with truncation and operator rescaling of the original measure and denote $\varrho_\varepsilon= \varepsilon^{-Q} T_\varepsilon\varrho$. Let $\mu_\varepsilon$ be the corresponding L\'evy measure, i.e.
   \begin{equation}\label{eq: mu_eps_def}
         \mu_\varepsilon(B)=\int_{\Sd} \int_0^{\varepsilon}\mathds{1}_B\{(r/\varepsilon)^Q\mb x\}   \varpi(d\mb x, dr),\quad B\in \mc B(\Rd).
    \end{equation}
   We show further that for measures $\mu_\varepsilon$, $\varepsilon\to 0$, the condition \eqref{eq: lem_converg_tightness_cond} follows from \eqref{eq: lem_converg_cond}.  Consequently,
   the vague convergence of the L\'evy measures $\mu_\varepsilon\overset{v}{\to}M$ is sufficient to establish weak convergence $\varrho_\varepsilon \overset{w}{\to}\mc D(Q,\sigma)$, $\varepsilon\to 0$, for a sequence $\varrho_\varepsilon$ obtained by the transformation of measure $\varrho$.
   We note that similar constructions have been considered in more detail in \cite{covo2009approximations} and \cite{grabchak2024smalljumpslevyprocesses}. The application of this result for simulations has also been discussed there.

\begin{lemma}\label{lem:  ID_convergence}
    Let $\mu_\varepsilon$, $\varepsilon>0$, be a collection of L\'evy measures given by \eqref{eq: mu_eps_def}. Assume also that there exists some probability measure $\sigma$ on $\Sd$ such that for every $t>0$ and every $D\in \mc B(S_Q)$, $\sigma(\partial D)=0$, it holds that
    \begin{equation*}
        \mu_\varepsilon(A_{t,D})\to -\log (1\wedge t)\sigma (D),\quad \varepsilon\to0,
    \end{equation*}
    where $A_{t,D}$ is given as in Lemma \ref{lem: vague_conv}. Then $\varrho_\varepsilon \overset{w}{\to}\mc D(Q,\sigma)$.
\end{lemma}

\begin{proof}
    Note the assumption above implies that $\mu_\varepsilon\overset{v}{\to} M$, where measure $M$ is given as in \eqref{eq: Lmeas_D_S}, see Lemma \ref{lem: vague_conv}. We will now show that 
    \begin{equation}\label{eq: mu_eps_tightness}
        \lim_{\delta \to 0} \limsup_{\varepsilon \to 0} \int_{|\mb x|<\delta}|\mb x|\mu_\varepsilon (d\mb x)=0.
    \end{equation}
    The idea of the proof is due to \cite[Prop. 1]{grabchak2024smalljumpslevyprocesses}.
    Rewrite now for some $\delta, \varepsilon>0$
    \begin{align*}
        \int_{|\mb x|<\delta}|\mb x|\mu_\varepsilon (d\mb x)=&\int_{\Sd} \int_0^{\varepsilon} |(r/\varepsilon)^Q\mb x| \mathds{1}_{B_\delta}((r/\varepsilon)^Q\mb x)  \varpi(d\mb x, dr)\\
        =&\int_{\Sd} \int_0^{1} |r^Q\mb x| \mathds{1}_{B_\delta}(r^Q\mb x)  \varpi(d\mb x, \varepsilon dr).
    \end{align*}

Using \eqref{eq: norm_exp_Q}, we get $|r^Q\mb x|\geq c_1r^{K_1}|\mb x|=c_1r^{K_1}$ for all $0\leq r\leq1$. Denote then $g(\delta)=(\delta/c_1)^{1/K_1}$ and notice that
\begin{align*}
        \int_{|\mb x|<\delta}|\mb x|\mu_\varepsilon (d\mb x)\leq& \int_{\Sd} \int_0^{g(\delta)} |r^Q\mb x| \varpi(d\mb x, \varepsilon dr)\\
        =&\sum_{n=1}^\infty \int_{\Sd} \int_{g(\delta)2^{-n}}^{g(\delta)2^{-n+1}} |r^Q\mb x| \varpi(d\mb x, \varepsilon dr)\\
        =&\sum_{n=1}^\infty \int_{\Sd} \int_{1/2 }^{1} \left |\left(\frac {rg(\delta)} {2^{n-1}} \right)^Q\mb x\right | \varpi\left(d\mb x, \varepsilon \frac {g(\delta)} {2^{n-1}} dr\right)\\
        \leq &\sum_{n=1}^\infty c_2 \left( \frac {g(\delta)} {2^{n-1}} \right)^{K_2} \int_{\Sd} \int_{1/2 }^{1} \left |r^Q\mb x\right | \varpi\left(d\mb x, \varepsilon \frac {g(\delta)} {2^{n-1}} dr\right),
    \end{align*}
    where we applied \eqref{eq: norm_exp_Q} with $s=g(\delta)/2^{n-1}<1$, assuming that $\delta$ is small enough.
    Notice now that 
    \begin{equation*}
        \int_{\Sd} \int_{1/2 }^{1} \left |r^Q\mb x\right | \varpi\left(d\mb x, \varepsilon \frac {g(\delta)} {2^{n-1}} dr\right)=\int_{\Rd}|\mb y| \mathds 1_{A_{1/2}^1}(\mb y)\mu_{\varepsilon \frac {g(\delta)} {2^{n-1}}}(d\mb y),
    \end{equation*}
    where $A_a^b=\{\mb y\in \Rd : \mb y=r^Q\mb x, r\in (a,b),\mb x\in \Sd\}$, $0\leq a<b<\infty$.
It follows from the vague convergence $\mu_\varepsilon\overset{v}{\to} M$, $\varepsilon\to 0$, that the last integral converges as $\varepsilon\to 0$ to 
\begin{equation*}
     \int_{\Rd}|\mb y| \mathds 1_{A_{1/2}^1}(\mb y) M(d\mb y)=m^*<\infty,
\end{equation*}
see also \cite[Thm. 1]{barczy2006portmanteau}.
Consequently, for any $\zeta>0$ there exists $\varepsilon^*>0$ such that for every $0<\varepsilon<\varepsilon^*$, it holds that
    \begin{equation*}
   \int_{\Rd}|\mb y| \mathds 1_{A_{1/2}^1}(\mb y)\mu_{\varepsilon \frac {g(\delta)} {2^{n-1}}}(d\mb y)< m^*+\zeta.
    \end{equation*}

The last estimate holds for all $n\geq 1$ and for $\delta$ small enough; indeed, since function $g(\delta)$ decreases to $0$ monotonically as $\delta\to 0$ it follows that $\varepsilon \frac {g(\delta)} {2^{n-1}}<\varepsilon<\varepsilon^*$ for $\delta$ small.
Consequently, we can deduce that for $\delta$ small, it holds that
\begin{align*}
        \int_{|\mb x|<\delta}|\mb x|\mu_\varepsilon (d\mb x)
        \leq c_2 (g(\delta))^{K_2}\sum_{n=1}^\infty 2^{-K_2(n-1)}  (m^*+\zeta).
    \end{align*}
We can now conclude that \eqref{eq: mu_eps_tightness} holds, which ends the proof by applying Theorem \ref{thm: convergence}.
\end{proof}

The next corollary provides a tractable example of the possible measure $\mu$ (and therefore corresponding $ID$ distribution) satisfying the conditions of Lemma \ref{lem:  ID_convergence}. It is a direct generalisation of \cite[Cor. 1]{grabchak2024smalljumpslevyprocesses}.
\begin{corollary}\label{cor: tr_conv}
    Consider $\varrho\in ID(\Rd)$ and assume that its L\'evy measure $\mu$, satisfying \eqref{eq: mu_Id_Levycond}, can be given by
        \begin{equation*}
        \mu(B)=\int_{\Sd} \int_0^\infty \mathds{1}_B\{r^Q\mb x\} \rho(\mb x, r) \sigma(d\mb x)dr,\quad B\in \mc B(\Rd),
    \end{equation*}
   i.e. assume that $\varpi(d\mb x,dr)=\rho(\mb x, r) \sigma(d\mb x)dr$ in \eqref{eq: ID_mu} for some function $\rho:\Sd\times (0,\infty)\mapsto[0,\infty)$ and a sigma-finite measure $\sigma$ on $\Sd$. 
    
     Assume also that for some $h: \Sd \mapsto  \R_+$, it holds that 
     \begin{gather*}
         \int\limits_{\Sd}h(\mb x)\sigma(d\mb x)=b<\infty \quad \text{and}\quad \lim\limits_{r\downarrow0}\int\limits_{\Sd}|r\rho(r,\mb x)-h(\mb x)|\sigma(d\mb x)=0.
     \end{gather*} 
     For simplicity, assume that $b=1$.
        Then we have
    \begin{equation*}
        \varrho_\varepsilon\overset{w}{\to} \mc D(Q, \tilde{\sigma}), \quad \varepsilon\to 0,
    \end{equation*}
    where  $\tilde{\sigma}(B)=\int\limits_B h(\mb x)\sigma(d\mb x)$, $B\in \mc B( \Sd)$, and $\varrho_\varepsilon$ is obtained from $\varrho$ by the transformation described above.
\end{corollary}

Note that we can also obtain the following simple corollary from Theorem \ref{thm: convergence}, which is an extension of Proposition \ref{prop: conv}\ref{prop: conv: iii}.

\begin{corollary}\label{cor: inf_conv}
    Assume that $\mb X_1, \mb X_2,\ldots $ is a sequence of random vectors such that $\mc L(\mb X_k)=\mc D(\frac 1 {w_k} Q,\nu_k),$ $k\geq 1$ with $\nu_k\in H$ and $w_k>0$, $\sum_{k=1}^\infty w_k=w<\infty$. Then the series $\sum_{k=1}^n \mb X_k$ converges weakly as $n\to \infty$ to $\mb X$ with $\mc L(\mb X)=\mc D(\frac 1 w Q,\tilde{\nu})$, where $\tilde{\nu}(\cdot)=\frac 1 w  \sum_{k=1}^\infty w_k\nu_k(\cdot)$.
\end{corollary}
\begin{proof}
    The proof is analogous to the proof of Proposition \ref{prop: conv}\ref{prop: conv: iii} if we notice that 
    \begin{equation*}
         \frac {w_1\nu_1+\cdots+ w_k\nu_k} {w}(D)\begin{cases}
             \leq \Tilde{\nu}(D), & k\in \mathbb N,\\
            \to \Tilde{\nu}(D) &  \text{as}\; k\to \infty.
         \end{cases}
    \end{equation*} 
    for all $D\subset \Sd$, and use Theorem \ref{thm: convergence}.
\end{proof}

\section{Simulations from the operator Dickman distribution}\label{sec: simul}

The random sample from the operator Dickman distribution can be simulated using the algorithm presented at the end of this section, where we use the representation \eqref{eq: opMDas_inf_sum}. Moreover, the sum in \eqref{eq: opMDas_inf_sum} is truncated at $N$-th term, where  $N=N_{max}\wedge K$ and $K=\inf\{k\in \mathbb{N}: \|(U_1\cdot\ldots\cdot U_k)^Q\|_\infty<\epsilon\}$ for some fixed thresholds $\epsilon>0$ and $N_{max} \in \mathbb{N}$. Here $\|Q\|_\infty$ denotes the largest absolute value among the entries of $Q\in \mc M_+$.

\begin{figure}[ht]
\centering
\includegraphics[width=0.31\linewidth,page=1]{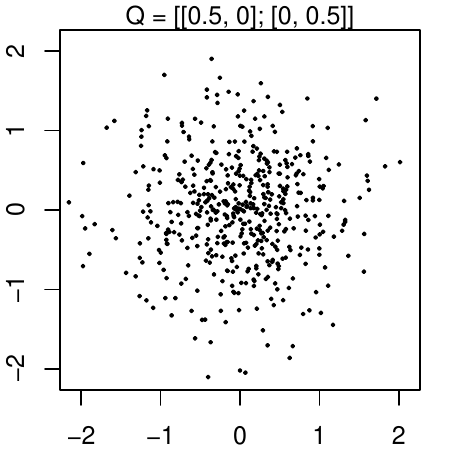}
\includegraphics[width=0.31\linewidth,page=2]{pictures/dickman-sph.pdf}
\includegraphics[width=0.31\linewidth,page=3]{pictures/dickman-sph.pdf}\\
\includegraphics[width=0.31\linewidth,page=4]{pictures/dickman-sph.pdf}
\includegraphics[width=0.31\linewidth,page=5]{pictures/dickman-sph.pdf}
\caption{The random samples of size 500 from the 2-dimensional operator Dickman distribution $\mc D(Q,\nu)$
with $\nu$ being the uniform distribution on the $\mathbb{S}^1$ and various~$Q$.}
\label{fig1}
\end{figure}

In Figure \ref{fig1}, we show random samples of size 500 from 
the 2-dimensional operator Dickman distribution $\mc D(Q, \nu)$ with different operators $Q\in \mc M_+$ and $\nu$ being the uniform distribution on $\mathbb{S}^1$.
The pictures demonstrate that the distribution $\mc D(Q, \nu)$ is rotationally invariant if $Q\in \mbb I$ and that the covariance matrix of a random vector from the $\mc D(Q, \nu)$ distribution depends on the matrix $Q$, cf. Proposition \ref{prop:mom_cov} and Proposition \ref{prop: conv} .

\begin{figure}[ht]
\includegraphics[width=0.31\linewidth,page=1]{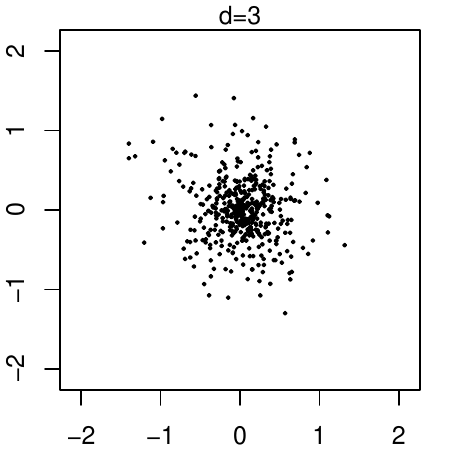}
\includegraphics[width=0.31\linewidth,page=2]{pictures/dickman-sph-d.pdf}
\includegraphics[width=0.31\linewidth,page=3]{pictures/dickman-sph-d.pdf}
\caption{The first and second coordinates of random samples of size 500 from the operator Dickman distribution $\mc D(I, \nu)$ on $\Rd$, where $\nu$
is the uniform distribution on the unit sphere $\Sd$, for $d=3,6,9$.}
\label{fig2}
\end{figure}

In Figure \ref{fig2}, we depict the first and second coordinates of 500
random $\Rd$-valued vectors sampled from $\mc D(I, \nu)$, 
$\nu$ being the uniform distribution on the unit sphere $\Sd$ for $d=3,6,9$.
We can see that the coordinates of the random samples are concentrated closer to the origin for bigger values of $d$.

Recall now that the von Mises distribution $\nu_{\kappa, \mu }$ on $\mathbb{S}^1$ is given as the distribution of the vector $(\cos\theta, \sin \theta)$ with the distribution of $\theta\in [0, 2\pi)$ given via the probability density function $f_{\kappa, \mu}(x)$
\begin{equation*}
    f_{\kappa, \mu}(x)=\frac{ \exp(\kappa \cos(x- \mu))}{2\pi I_0(\kappa)},\quad x \in [0,2\pi). 
\end{equation*}
Here $I_0(\cdot)$ is the modified Bessel function of the first kind of order $0$, and $\kappa>0$, $\mu \in [0,2\pi)$ are the parameters. 
\begin{figure}[ht]
\centering
\includegraphics[width=0.31\linewidth,page=1]{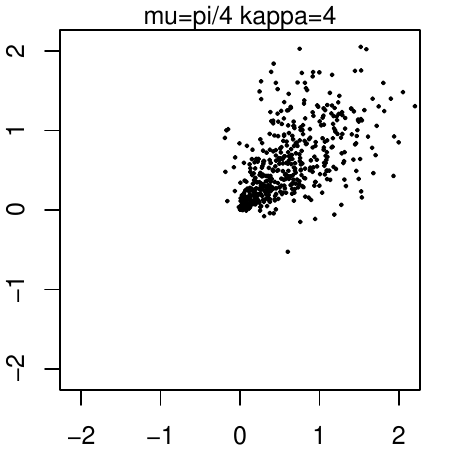}
\includegraphics[width=0.31\linewidth,page=2]{pictures/dickman-vonmises.pdf}
\includegraphics[width=0.31\linewidth,page=3]{pictures/dickman-vonmises.pdf}
\caption{The random samples of size 500 from the $2$-dimensional Dickman distribution $\mc D(I, \nu_{\mu,\kappa})$
with various values of parameters $\mu$ and $\kappa$.}
\label{fig3}
\end{figure}

In Figure \ref{fig3}, we show the random samples of size 500 from the $2$-dimensional Dickman distribution $\mc D(I,\nu_{\kappa,\mu})$ 
with various values of parameters $\mu$ and $\kappa$. We can see that the points are distributed 
along the direction specified by the angle  $\mu$, and 
the dispersion of points around this direction depends on $\kappa$.

\begin{algorithm}[ht]
\caption{Sampling from the operator Dickman distribution.}
\begin{algorithmic}[1]
\Require $n$, $Q \in \mc M_+$, distribution $\nu$ on $\Rd$, tolerance $\varepsilon > 0$, cap $N_{\max}$
\Ensure $ X^{(1)}, \dots,  X^{(n)} \in \mathbb{R}^d$

\Function{rDickman}{$n,Q,\epsilon,N_{\max}$}
    \For{$s=1$ to $n$}
        \State $x \gets 0_d$, $M \gets I_d$
        \Repeat
           \State $u \sim U[0,1]$
            \State $M \gets$ $e^{Q\log u}\cdot M$
            \State $x \gets x + MW$, with $W \sim \nu$ 
            \State $terms \gets terms+1$
        \Until{$\|M\|_\infty < \epsilon$ or terms $=N_{\max}$}
        \State $X^{(s)} \gets x$
    \EndFor
    \State \Return $\{X^{(s)}\}$
\EndFunction

\end{algorithmic}
\end{algorithm}

\begin{center}
    \textbf{Acknowledgements}
\end{center}
Nikolai Leonenko (NL) would like to thank for support the ARC Discovery Grant DP220101680 (Australia), Croatian Scientific Foundation (HRZZ) grant “Scaling in Stochastic Models” (IP-2022-10-8081), grant FAPESP 22/09201-8 (Brazil) and the Taith Research Mobility grant (Wales, Cardiff University).

\bibliography{literature} 
\end{document}